\documentclass[leqno]{amsart}

\usepackage{stmaryrd,graphicx,mathrsfs}
\usepackage{amssymb,amsmath,amscd,amsthm,amsfonts}
\usepackage{float}
\usepackage[all,cmtip]{xy}
\DeclareMathAlphabet{\mathpzc}{OT1}{pzc}{m}{it}
\usepackage{txfonts,pxfonts,latexsym,wasysym}
\usepackage[pdftex]{hyperref}
\author{Jeremy Brazas}

\address{Department of Mathematics \& Statistics; Georgia State University, Atlanta, Georgia 30303}

\title{Generalized covering space theories}


\newtheorem{theorem}{Theorem}
\newtheorem{lemma}[theorem]{Lemma}
\newtheorem{proposition}[theorem]{Proposition}
\newtheorem{corollary}[theorem]{Corollary}
\theoremstyle{definition}\newtheorem{definition}[theorem]{Definition}

\theoremstyle{definition}\newtheorem{example}[theorem]{Example}
\theoremstyle{definition}
\theoremstyle{definition}
\theoremstyle{definition}\newtheorem{remark}[theorem]{Remark}
\newcommand{\ds}{\displaystyle}
\newcommand{\wt}{\widetilde}

\newcommand{\bbr}{\mathbb{R}}

\newcommand{\bbh}{\mathbb{H}}

\newcommand{\oX}{\overline{X}}
\newcommand{\ox}{\overline{x}}

\newcommand{\hX}{\widehat{X}}
\newcommand{\hx}{\hat{x}}

\newcommand{\hf}{\widehat{f}}
\newcommand{\tX}{\widetilde{X}}
\newcommand{\tY}{\widetilde{Y}}
\newcommand{\ty}{\tilde{y}}
\newcommand{\ui}{[0,1]}

\newcommand{\tx}{\tilde{x}}

\newcommand{\tXh}{\widetilde{X}_{H}}
\newcommand{\txh}{\tilde{x}_{H}}

\newcommand{\pxxo}{P(X,x_0)}
\newcommand{\pionex}{\pi_{1}(X,x_0)}

\newcommand{\mcc}{\mathcal{C}}
\newcommand{\mcd}{\mathcal{D}}

\newcommand{\scrh}{\mathscr{H}}

\newcommand{\topz}{\mathbf{Top_{0}}}

\newcommand{\dcovx}{\mathbf{DCov}(X)}

\newcommand{\btopz}{\mathbf{bTop_{0}}}
\newcommand{\delcovx}{\Delta\mathbf{Cov}(X)}

\newcommand{\ccov}{\mathbf{Cov}_{\mcc}}
\newcommand{\ccovx}{\mathbf{Cov}_{\mcc}(X)}

\newcommand{\hccovx}{\mathbf{Cov}_{\hullc}(X)}
\newcommand{\hccovcx}{\mathbf{Cov}_{\hullc}(c(X))}

\newcommand{\hullc}{\scrh(\mcc)}
\newcommand{\hulld}{\scrh(\mcd)}
\newcommand{\lpcz}{\mathbf{lpc_0}}
\newcommand{\fan}{\mathbf{Fan}}
\newcommand{\mcf}{\mathcal{F}}
\begin{document}
\begin{abstract}
In this paper, we unify various approaches to generalized covering space theory by introducing a categorical framework in which coverings are defined purely in terms of unique lifting properties. For each category $\mathcal{C}$ of path-connected spaces having the unit disk as an object, we construct a category of $\mathcal{C}$-coverings over a given space $X$ that embeds in the category of $\pi_1(X,x_0)$-sets via the usual monodromy action on fibers. When $\mathcal{C}$ is extended to its coreflective hull $\mathscr{H}(\mathcal{C})$, the resulting category of based $\mathscr{H}(\mathcal{C})$-coverings is complete, has an initial object, and often characterizes more of the subgroup lattice of $\pi_1(X,x_0)$ than traditional covering spaces.

We apply our results to three special coreflective subcategories: (1) The category of $\Delta$-coverings employs the convenient category of $\Delta$-generated spaces and is universal in the sense that it contains every other generalized covering category as a subcategory. (2) In the locally path-connected category, we preserve notion of generalized covering due to Fischer and Zastrow and characterize the topology of such coverings using the standard whisker topology. (3) By employing the coreflective hull $\mathbf{Fan}$ of the category of all contractible spaces, we characterize the notion of continuous lifting of paths and identify the topology of $\mathbf{Fan}$-coverings as the natural quotient topology inherited from the path space.
\end{abstract}
\maketitle

\section{Introduction}
When a topological space $X$ is path-connected, locally path-connected, and semilocally simply-connected, the entire subgroup lattice of $\pionex$ can be understood in terms of the covering spaces of $X$ \cite{Spanier66}. More precisely, the monodromy functor $\mu:\mathbf{Cov}(X)\to\pionex\mathbf{Set}$ from the category of coverings over $X$ to the category of $\pionex$-sets is an equivalence of categories. If $X$ lacks a simply connected covering space, more sophisticated machinery is often needed to understand the combinatorial structure of $\pionex$; however, significant advancements have been made in the past two decades. In this paper, we develop a categorical framework which unifies various attempts to generalize covering space theory, clarifies their relationship to classical topological constructions, and illuminates the theoretical extent to which such methods provide information about the fundamental group.

Many authors have attempted to extend the covering-theoretic approach to more general spaces, e.g. \cite{BP07uniform, Brazsemi,BDLM10uniform, Foxonshape72,Lub62}. The usefulness of one generalization over another depends on the intended application. For instance, Fox's overlays \cite{Foxonshape72} provide no more information about the subgroup lattice of $\pionex$ than traditional covering maps but admit a much more general classification in terms of the fundamental pro-group. Semicoverings \cite{Brazsemi} are intimately related to topological group structures on fundamental groups \cite{Brazoverlay,FZ13} and have natural applications to general topological group theory \cite{BrazNS}. In the current paper, we consider maps defined purely in terms of unique lifting properties. Our definitions are inspired by the initial approach of Hanspeter Fischer and Andreas Zastrow in \cite{FZ07} and the subsequent papers \cite{BDLM08} and \cite{Dydak11}.

Generalized covering maps defined in terms of unique lifting properties often exist when standard covering maps do not and provide combinatorial information about fundamental groups of spaces which are not semilocally simply connected \cite{FRVZ11,FZ07}. For instance, one-dimensional spaces such as the Hawaiian earring, Menger curve, and Sierpinski carpet admit certain generalized ``universal" coverings having the structure of topological $\bbr$-trees (called universal $\lpcz$-coverings in the current paper) on which $\pionex$ acts by homeomorphism. These $\bbr$-trees behave like generalized Caley graphs \cite{FZ013caley} and have produced an explicit word calculus for the fundamental group of the Menger curve \cite{FZ14menger} in which the fundamental groups of all other one-dimensional and planar Peano continua embed.

In the current paper, we begin with a category $\mcc$ of path-connected spaces having the unit disk as an object and define $\mcc$-coverings to have a unique lifting property with respect to maps on the objects of $\mcc$. In Section 2, we introduce categories of $\mcc$-coverings and explore their properties. In particular, we show the category $\ccovx$ of $\mcc$-coverings over a given space $X$ canonically embeds into $\pionex\mathbf{Set}$ by a fully faithful monodromy functor $\mu:\ccovx\to\pionex\mathbf{Set}$. Thus a $\mcc$-covering $p:\tX\to X$ is completely characterized up to isomorphism by the conjugacy class of the stabilizer subgroup $H=p_{\ast}(\pi_1(\tX,\tx_0))$. The category of based $\mcc$-coverings becomes highly structured - exhibiting many properties that categories of classical covering maps lack - when we take $\mcc$ to be the coreflective hull of a category of simply connected spaces.

The literature on fundamental groups of wild spaces and generalized covering spaces primarily takes the viewpoint of considering fundamental groups at a chosen basepoint and applying the relevant infinite group theoretic concepts. In some situations, it may be preferable to avoid picking a basepoint and work with the fundamental groupoid $\pi_1(X)$. In Section 3, we relate our categorical treatment of generalized coverings to covering morphisms of groupoids \cite{Brown06}.

In Section 4, we identify the convenient (in the sense of \cite{BT,Steenrod}) category of $\Delta$-generated spaces, used in directed topology \cite{FRconvenientfordirectedhom} and diffeology \cite{CSWdtopology} as the setting for a ``universal" theory of generalized covering maps. The category of $\Delta$-coverings is universal in the sense that any other category of $\mcc$-coverings canonically embeds within it. In a more practical sense, any attempt to characterize the subgroup structure of $\pionex$ using maps having unique lifting of paths and homotopies of paths is retained as a special case of $\Delta$-coverings.

Since we define $\mcc$-coverings only in terms of abstract unique lifting properties, we are left with two important questions for given $\mcc$.
\begin{enumerate}
\item[] \textbf{Structure Question:} If a $\mcc$-covering exists, is there a simple characterization of the topology of the $\mcc$-covering space $\tX$?\\
\item[] \textbf{Existence Question:} Monodromy $\mu:\ccovx\to\pionex\mathbf{Set}$ is fully faithful but need not be essentially surjective. Thus a subgroup $H\subseteq \pionex$ is said to be a $\mcc$\textit{-covering subgroup} if there exists a corresponding $\mcc$-covering $p:\tX\to X$ and $\tx_0\in\tX$ such that $H=p_{\ast}(\pi_1(\tX,\tx_0))$. Is there a practical characterization of the $\mcc$-covering subgroups of $\pionex$?
\end{enumerate}

In general, the Existence Question is challenging and must be taken on a case-by-case basis depending on which category $\mcc$ is being used. Some necessary and sufficient conditions for the existence of $\mcc$-coverings are known for the locally path-connected category $\mcc=\lpcz$ \cite{BF15,BDLM08,FRVZ11,FZ07}; however, general characterizations remain the subject of ongoing research beyond the scope of the current paper. Nevertheless, it is reasonable to expect that an answer to the Structure Question will provide some help in answering the Existence Question.

In Sections 5 and 6, we answer the Structure Question in two important cases. First, we show the topology of $\lpcz$-coverings must always be the so-called ``standard" or ``whisker" topology. Moreover, we show that if $X$ is first countable, then the notions of $\lpcz$-covering and $\Delta$-covering over $X$ agree. Second, we consider the coreflective hull $\mathbf{Fan}$ of all contractible spaces, which is generated by so-called \textit{directed arc-fan spaces}. We identify the topology of $\mathbf{Fan}$-coverings, which characterize the notion of continuous lifting of paths \cite{Brazsemi}, as the natural quotient topology inherited from the path space. The Existence Question for $\mathbf{Fan}$-coverings remains open and is likely to have application to topological group theory.
\section{Generalized covering theories}
\subsection{Notational considerations}

Throughout this paper, $X$ will denote a path-connected topological space with basepoint $x_0\in X$. We take $\topz$ to be the categories of path-connected topological spaces and maps (i.e. continuous function) and $\btopz$ to be the category of based path-connected spaces and based maps. All subcategories of $\topz$ and $\btopz$ considered in this paper are assumed to be full subcategories. Given a subcategory $\mcc\subset \topz$, the corresponding category of based spaces $(X,x)$ where $X\in\mcc$ is denoted $\mathbf{b}\mcc$. If $f:(X,x)\to(Y,y)$ is a based map, $f_{\ast}:\pi_1(X,x)\to\pi_1(Y,y)$ will denote the homomorphism induced by $f$ on fundamental groups.

Let $\ui$ denote the unit interval and $D^2=\{(x,y)\in\bbr^2|x^2+y^2\leq 1\}$ the closed unit disk with basepoint $d=(1,0)$. A \textit{path} in a space $X$ is a continuous function $\alpha:[0,1]\to X$. If $\alpha:\ui\to X$ is a path, then $\alpha^{-}(t)=\alpha(1-t)$ is the reverse path. If $\alpha,\beta:\ui\to X$ are paths such that $\alpha(1)=\beta(0)$, then $\alpha\cdot\beta$ denotes the usual concatenation of paths. The constant path at $x\in X$ is denoted by $c_x$.
\begin{definition}\label{upldef}
A map $f:X\to Y$ has the \textit{unique path lifting property} if for any two paths $\alpha,\beta:[0,1]\to X$, we have $\alpha=\beta$ whenever $f\circ \alpha=f\circ \beta$ and $\alpha(0)=\beta(0)$.
\end{definition}
Let $P(X)$ denote the space of paths in $X$ with the compact-open topology. The compact-open topology of $P(X)$ is generated by the subbasic sets $\ds\langle K,U\rangle =\{\alpha\in P(X)|\alpha(K)\subseteq U\}$ where $K\subset \ui$ is compact and $U\subset X $ is open. For given $x\in X$, let $P(X,x)=\{\alpha\in P(X)|\alpha(0)=x\}$ denote the subspace of paths which start at $x$ and $\Omega(X,x)=\{\alpha\in P(X)|\alpha(0)=x=\alpha(1)\}$ denote the subspace of loops based at $x$. A map $f:(X,x)\to (Y,y)$ induces a continuous function $f_{\#}:P(X,x)\to P(Y,y)$ given by $f_{\#}(\alpha)=f\circ \alpha$. Note that $f:X\to Y$ has the unique path lifting property if and only if $f_{\#}:P(X,x)\to P(Y,p(x))$ is injective for every $x\in X$.
\subsection{Disk-coverings}

In the attempt to minimize the conditions one might impose on a covering-like map $p:E\to X$, we are led to the following definition due to Dydak \cite{Dydak11}. The definition is minimal in the sense that our goal is to retain information about the traditional fundamental group $\pionex$ and to do this one should require unique lifting of all paths and homotopies of paths.
\begin{definition}\label{diskcoveringdef}
A map $p:E\to X$ is a \textit{disk-covering} if $E$ is non-empty, path-connected and if for every $e\in E$ and map $f:(D^2,d)\to(X,p(e))$, there is a unique map $\hf:(D^2,d)\to(E,e)$ such that $p\circ\hf=f$.
\end{definition}
Since the unit interval is a retract of $D^2$, it is clear that if $p:E\to X$ is a disk-covering, then every path $\alpha:(\ui,0)\to (X,p(e))$ also has a unique lift $\wt{\alpha}_{e}:(\ui,0)\to (E,e)$ such that $p\circ\wt{\alpha}_{e}=\alpha$. It follows that $p$ must be surjective. The induced homomorphism $p_{\ast}:\pi_1(E,e)\to \pi_1(X,p(e))$ is injective for every $e\in E$ and we have $[\alpha]\in p_{\ast}(\pi_1(E,e))$ if and only if the unique lift $\wt{\alpha}_{e}:(\ui,0)\to (E,e)$ such that $p\circ\wt{\alpha}_{e}=\alpha$ is a loop.

A morphism of disk-coverings $p:E\to X$ and $q:E'\to X'$ is a pair $(f,g)$ of maps $f:E\to E'$ and $g:X\to X'$ such that $g\circ p=q\circ f$. Let $\mathbf{DCov}$ denote the category of disk-coverings and for a given space $X$, let $\dcovx$ be the subcategory of disk-coverings over $X$ where morphisms are pairs $(f,id)$, that is, commuting triangles:
\[\xymatrix{
E \ar[dr]_-{p} \ar[rr]^-{f} && E' \ar[dl]^-{q}\\ & X  }\]

If $G=\pionex$, then traditional arguments from classical covering space theory imply the existence of a canonical ``monodromy" functor $\mu:\dcovx\to G\mathbf{Set}$ to the category $G\mathbf{Set}$ of $G$-Sets (sets $A$ with a group action $(g,a)\mapsto g\cdot a$) and $G$-equivariant functions (functions $f:A\to B$ satisfying $f(g\cdot a)=g\cdot f(a)$). On objects, $\mu$ is defined as the fiber $\mu(p)=p^{-1}(x_0)$. If $q:E'\to X$ is a disk-covering and $f:E\to E'$ is a map such that $q\circ f=p$, then $\mu(f)$ is the restriction of $f$ to a $G$-equivariant function $p^{-1}(x_0)\to q^{-1}(x_0)$.
\begin{lemma}
\label{faithfulpropdiskcov}
If $(X,x_0)\in\btopz$ and $G=\pionex$, the functor $\mu:\dcovx\to G\mathbf{Set}$ is faithful.
\end{lemma}
\begin{proof}
Suppose $p:E\to X$ and $q:E'\to X$ are disk-coverings over $X$. Let $f,g:E\to E'$ be maps such that $q\circ f=p=q\circ f$ and $\mu(f)=\mu(g)$ as functions $p^{-1}(x_0)\to q^{-1}(x_0)$. To see that $\mu$ is faithful, we check that $f=g$. Fix $e_0\in p^{-1}(x_0)$, pick a point $e\in E$, and a path $\gamma:\ui\to E$ from $e_0$ to $e$. By assumption, we have $f(e_0)=g(e_0)=e_{0}'$ for some point $e_{0}'\in q^{-1}(x_0)$. If $\beta:\ui\to E'$ is the unique path such that $q\circ\beta=p\circ\gamma$ and $\beta(0)=e_{0}'$, then $f\circ\gamma=\beta=g\circ\gamma$ by unique path lifting. In particular, $f(e)=\beta(1)=g(e)$ and thus $f=g$.
\end{proof}
The functor $\mu:\dcovx\to G\mathbf{Set}$ will not typically be full since there are non-isomorphic disk-coverings corresponding to isomorphic group actions (see Example \ref{nonequivdiskcoveringexample} below). In this sense, disk coverings - despite their great generality - are not an ideal candidate for a generalized covering theory. We will often restrict $\mu$ to a subcategory $\mcd\subseteq \dcovx$. When we do this, we will still use the symbol $\mu$ to represent the restriction functor $\mcd\to G\mathbf{Set}$.

The subgroups $H\subseteq G$ which arise as the stabilizer subgroups of $G$-sets in the image of $\mu$ are precisely those for which there is a disk covering $p:(E,e)\to (X,x)$ with $p_{\ast}(\pi_1(E,e))=H$. Since some subgroups of $\pionex$ need not arise as such stabilizers, we give special attention to those that do.
\begin{definition}
A subgroup $H\subseteq\pionex$ is a \textit{disk-covering subgroup} if there is a disk covering $p:(E,e)\to (X,x)$ with $p_{\ast}(\pi_1(E,e))=H$.
\end{definition}
Another useful property not held by classical covering maps is the following 2-of-3 lemma. We leave the straightforward proof to the reader.
\begin{lemma}\label{twoofthreediskcovlemma}
Suppose $p:E\to X$ and $q:E'\to E$ are maps. If two of the maps $p,q,p\circ q$ are disk-coverings, then so is the third.
\end{lemma}
\subsection{$\mcc$-coverings and their properties}
The following definition is based on the definition of generalized covering in \cite{FZ07} but, in the spirit of \cite{Dydak11}, allows for a wider range of possible lifting criteria.
\begin{definition}\label{ccovdef}
Let $\mcc\subseteq \topz$ be a full subcategory of non-empty, path-connected spaces having the unit disk $D^2$ as an object. A $\mcc$\textit{-covering map} is a map $p:\tX\to X$ such that
\begin{enumerate}
\item $\tX\in \mcc$,
\item For every space $Y\in\mcc$, point $\tx\in\tX$, and based map $f:(Y,y)\to(X,p(\tx))$ such that $f_{\ast}(\pi_1(Y,y))\subseteq p_{\ast}(\pi_1(\tX,\tx))$, there is a unique map $\wt{f}:(Y,y)\to (\tX,\tx)$ such that $p\circ\wt{f}=f$.
\end{enumerate}
We call $p$ a \textit{universal} $\mcc$\textit{-covering} if $\tX$ is simply connected. Furthermore, we call $p$ a \textit{weak }$\mcc$\textit{-covering map} if $p$ only satisfies condition 2.
\end{definition}
\begin{remark}
The second condition in the definition of $\mcc$-covering is reminiscent of the unique lifting criterion used in classical covering space theory. Since $D^2$ is an object of $\mcc$ and $\pi_1(D^2,d)=1$, it is clear that every weak $\mcc$-covering is a disk-covering and is therefore surjective. In general, the lift $\tilde{f}:(Y,y_0)\to (\tX,\tx_0)$ in condition (2) can be described as follows: Let $y\in Y$ and $\gamma:\ui\to Y$ be any path from $y_0$ to $y$. Then $\tilde{f}(y)$ is the endpoint of the unique lift $\wt{f\circ\gamma}_{\tx_0}:\ui\to\tX$ starting at $\tx_0$.\\
\indent If $\mcc$ is a category of simply connected spaces, the notion of weak $\mcc$-covering agrees with the maps of study in \cite{Dydak11}. For instance, if $\mcd$ is the category whose only object is $D^2$, a weak $\mcd$-covering is precisely a disk-covering. In general, the ``weak" coverings are not unique up to homeomorphism and, as we will show, can always be replaced by some category of genuine $\mcc$-coverings without losing information about $\pionex$.
\end{remark}
Since we always assume $D^2$ is an object of $\mcc$, we have the following implications for a given map $p:\tX\to X$:
\[\txt{$p$ is a $\mcc$-covering $\Rightarrow p$ is a weak $\mcc$-covering $\Rightarrow p$ is a disk-covering}\]

Let $\ccov$ denote the category of $\mcc$-coverings and $\ccovx$ denote the category of $\mcc$-coverings over $X$; we view these as full subcategories of $\mathbf{DCov}$ and $\dcovx$ respectively. Since every $\mcc$-covering $p:\tX\to X$ is a disk-covering, we may apply the monodromy functor $\mu$ to $p$ to obtain the corresponding group action of $G=\pionex$ on the fiber $p^{-1}(x_0)$. The following embedding theorem illustrates that $\mcc$-coverings, unlike disk-coverings, are characterized up to isomorphism by this group action.

\begin{theorem}\label{fullyfaithfulthm}
The functor $\mu:\ccovx\to G\mathbf{Set}$ is fully faithful.
\end{theorem}
\begin{proof}
Since $\mu:\dcovx\to G\mathbf{Set}$ is faithful by Lemma \ref{faithfulpropdiskcov}, the restriction is also faithful. To check that $\mu$ is full, suppose $p:E\to X$ and $q:E'\to X$ are $\mcc$-coverings and that $f:p^{-1}(x_0)\to q^{-1}(x_0)$ is a $G$-equivariant function. Being $G$-equivariant means that $f$ satisfies the equation $f(\widetilde{\alpha}_{e}(1))=\widetilde{\alpha}_{f(e)}(1)$ for every loop $\alpha\in\Omega(X,x_0)$ and point $e\in p^{-1}(x_0)$. Fix $e_0\in p^{-1}(x_0)$, let $e_{0}'=f(e_0)$, and consider a loop $\beta\in\Omega(E,e_0)$. We have
$$e_{0}'=f(e_0)=f\left(\wt{\alpha}_{e_0}(1)\right)=\wt{p\circ\beta}_{e_0}(1)$$which implies that $p\circ\beta$ lifts to a loop in $E'$ based at $f(e_0)=e_{0}'\in q^{-1}(x_0)$. Thus $p_{\ast}([\beta])\in q_{\ast}(\pi_1((E',e_{0}')))$. Since $p_{\ast}(\pi_1(E,e_0))\subseteq q_{\ast}(\pi_1((E',e_{0}')))$, there is a unique morphism $\wt{p}:(E,e_0)\to (E',e_{0}')$ such that $q\circ \wt{p}=p$. It suffices to check that $f$ is the restriction of $\wt{p}$ to $p^{-1}(x_0)$. Let $e\in p^{-1}(x_0)$ and $\gamma:\ui\to E$ be a path from $e_0$ to $e$. If $\alpha=p\circ\gamma$, then $$f(e)=f(\wt{\alpha}_{e_0}(1))=\wt{\alpha}_{e_{0}'}(1)=\wt{p}(e).$$
\end{proof}
The functor $\mu$ in Theorem \ref{fullyfaithfulthm} is not necessarily an equivalence of categories since subgroups $H\subseteq \pionex$ exist for which there may be no $\mcc$-covering $p:(\tX,\tx_0)\to (X,x_0)$ such that $H=p_{\ast}(\pi_1(\tX,\tx_0))$: see Examples \ref{failureexample} and \ref{failureexample2}.
\begin{definition}
A subgroup $H\subseteq \pionex$ is a $\mcc$\textit{-covering subgroup} if there exists a $\mcc$-covering map $p:(\tX,\tx)\to (X,x_0)$ such that $p_{\ast}(\pi_1(\tX,\tx))=H$.
\end{definition}
\begin{remark}\label{conjugateremark}
By changing the basepoint of the path-connected space $\tX$ in the fiber $p^{-1}(x_0)$, it is clear that whenever $H$ is a $\mcc$-covering subgroup, every conjugate of $H$ is also a $\mcc$-covering subgroup.
\end{remark}
The following proposition is straightforward to verify based on arguments from traditional covering space theory. In particular, it confirms that $\ccov$ is closed under the operation of composition, a property not generally held by covering maps in the classical sense.
\begin{proposition}\label{twoofthreeprop}
Suppose $p:\tX\to X$ and $q:\tY\to\tX$ are maps.
\begin{enumerate}
\item If $p$ and $q$ are $\mcc$-coverings, then so is $p\circ q$,
\item If $p$ and $p\circ q$ are $\mcc$-coverings, then so is $q$.
\end{enumerate}
\end{proposition}
\subsection{Coreflections}
Throughout this section, we assume $\mcc$ is a category having $D^2$ as an object. We typically want $\mcc$ to have ``enough" objects to provide an interesting theory. Certainly, if the only object of $\mcc$ is $D^2$, then the category $\ccov$ is rather uninteresting; we'd even have $\ccov([0,1])=\emptyset$. In this section, we will show that replacing any $\mcc$ with its coreflective hull in $\topz$ provides a richer theory of coverings without sacrificing any monodromy data.
\begin{definition}
The \textit{coreflective hull} of $\mcc$ in $\topz$ is the full subcategory $\hullc$ of $\topz$ consisting of all path-connected spaces which are the quotient of a topological sum of objects of $\mcc$.
\end{definition}
Certainly $\hullc$ is closed under quotients and $\mcc\subseteq\hullc$.
\begin{example}
Well-known examples of coreflective hulls taken within $\topz$ include the following:
\begin{enumerate}
\item If $\mcd$ is the category whose only object is $D^2$, then $\hulld$ is the category $\Delta\topz$ of so-called $\Delta$-generated spaces\footnote{The term ``$\Delta$-generated" comes from the fact that $\Delta\topz$ is the coreflective hull of the category consisting of the standard n-simplices $\Delta_n$.} The category of $\Delta$-generated spaces has been used in directed topology \cite{FRconvenientfordirectedhom} and diffeology \cite{CSWdtopology}. We treat the $\Delta$-generated category in more detail in Section \ref{deltacoveringsection}.
\item The category $\lpcz$ of path-connected, locally path-connected spaces is it's own coreflection since the quotient of every locally path-connected space is locally path-connected. More practically, $\lpcz$ is the coreflective hull of the category of directed arc-hedgehogs: see Lemma \ref{hedgehogislpc} below.
\item The category $\topz$ is itself the coreflective hull of the category $\mathbf{Top_1}$ of all simply connected spaces. The proof is a nice exercise and is left to the reader.
\end{enumerate}
\end{example}
\begin{proposition}
If $D^2$ is an object of $\mcc$, then $\Delta\topz\subseteq\hullc$. In particular, $\ui$ is an object of $\hullc$.
\end{proposition}
The category $\hullc$ is coreflective in the sense that the inclusion functor $\hullc\to\topz$ has a right adjoint $c:\topz\to\hullc$ where $c(X)$ has the quotient (i.e. final) topology with respect to all maps $g:Y\to X$ with $Y\in\mcc$. A set $U\subseteq X$ is $\mcc$\textit{-open in} $X$ if for every map $f:Z\to X$ where $Z\in\mcc$, $f^{-1}(U)$ is open in $Z$. A set $U$ is open in $c(X)$ if and only if $U$ is $\mcc$-open in $X$. The fact that $c:\topz\to\hullc$ is right adjoint is equivalent to the following more practical formulation.
\begin{proposition}\label{coreflectionbasicprop}
The identity function $id:c(X)\to X$ is continuous. Moreover, if $g:Y\to X$ is continuous where $Y\in\hullc$, then $g:Y\to c(X)$ is also continuous.
\end{proposition}
Since $\ui$ and $D^2$ are objects of $\hullc$, every path and homotopy of paths in $X$ is also continuous with respect to the topology of $c(X)$. Thus the continuous identity $id:c(X)\to X$ is a disk-covering, which induces an isomorphism $\pi_1(c(X),x)\to\pi_1(X,x)$ on fundamental groups for every $x\in X$.
\begin{corollary}
For every $X\in\topz$, the identity function $id:c(X)\to X$ is an $\hullc$-covering.
\end{corollary}
\begin{proof}
By construction, $c(X)\in\hullc$. Suppose $Y\in \hullc$ and $f:(Y,y)\to (X,x)$ is a based map such that $f_{\ast}(\pi_1(Y,y))\subseteq (id)_{\ast}(\pi_1(c(X),x))=\pi_1(X,x)$. Since $Y\in\hullc$, $f:Y\to c(X)$ is also continuous and certainly satisfies $f\circ id=f$.
\end{proof}
The next lemma tells us that every weak $\mcc$-covering induces a $\hullc$-covering which retains identical monodromy data of $\pionex$.
\begin{lemma}\label{ccovtohullccovlemma}
If $p:\tX\to X$ is a weak $\mcc$-covering, then $\psi(p):c(\tX)\to X$ is a $\hullc$-covering. Moreover, the morphism\[\xymatrix{
c(\tX) \ar[dr]_-{\psi(p)} \ar[rr]^-{id} && \tX \ar[dl]^-{p} \\ & X
}\]of disk-coverings over $X$ induces an isomorphism $\mu(\psi(p):c(\tX)\to X)\cong\mu(p:\tX\to X)$ of $G$-sets.
\end{lemma}
\begin{proof}
First, we check that $\psi(p)$ is a $\hullc$-covering. By construction, $c(\tX)\in\hullc$. We check that $\psi(p)$ is a weak $\hullc$-covering. Suppose $\tx\in\tX$, $p(\tx)=x$, and $f:(Y,y)\to (X,x)$ is a based map where $Y\in\hullc$ and $f_{\ast}(\pi_1(Y,y))\subseteq (\psi(p))_{\ast}(\pi_1(c(\tX),\tx))$. Since the continuous identity function $c(\tX)\to \tX$ induces an isomorphism of fundamental groups, we have $(\psi(p))_{\ast}(\pi_1(c(\tX),\tx))=p_{\ast}(\pi_1(\tX,\tx))$.

Define a function $\wt{f}:(Y,y)\to(\tX,\tx)$ as follows: for $z\in Y$, pick a path $\gamma:\ui\to Y$ from $y$ to $z$. If $\wt{f\circ\gamma}:\ui\to\tX$ is the unique lift such that $p\circ \wt{f\circ\gamma}=f\circ \gamma$ and $\wt{f\circ\gamma}(0)=\tx$, we let $\wt{f}(z)=\wt{f\circ\gamma}(1)$. Since $f_{\ast}(\pi_1(Y,y))\subseteq p_{\ast}(\pi_1(\tX,\tx))$ and $p$ is a disk-covering, it is clear that $\wt{f}$ is a well-defined function and is unique. It suffices to show $\wt{f}:Y\to c(\tX)$ is continuous with respect to the topology of $c(\tX)$. Suppose $U\subseteq c(\tX)$ is open. Since $Y\in\hullc$, we need to check that $g^{-1}(\wt{f}^{-1}(U))$ is open in $Z$ for every map $g:Z\to Y$ where $Z\in\mcc$. Suppose we have a point $z\in g^{-1}(\wt{f}^{-1}(U))$. Pick any path $\gamma$ from $y$ to $g(z)$ so that $\wt{f}(g(z))=\wt{f\circ\gamma}(1)=\tx '$. Note that
\begin{eqnarray*}
(f\circ g)_{\ast}(\pi_1(Z,z)) &\subseteq & f_{\ast}(\pi_1(Y,g(z)))\\
&=& [f\circ \gamma]^{-1}f_{\ast}(\pi_1(Y,y))[f\circ \gamma]\\
&\subseteq & [f\circ \gamma]^{-1}p_{\ast}(\pi_1(\tX,\tx))[f\circ \gamma]\\
&= & p_{\ast}\left(\left[\wt{f\circ \gamma}\right]^{-1}\pi_1(\tX,\tx)\left[\wt{f\circ \gamma}\right]\right)\\
&=& p_{\ast}(\pi_1(\tX,\tx '))
\end{eqnarray*}
Since $Z\in\mcc$ and $p:\tX\to X$ is a weak $\mcc$-covering, there is a unique map $k:(Z,z)\to(\tX,\tx ')$ such that $p\circ k=f\circ g$. Since $Z\in\hullc$, the function $k:Z\to c(\tX)$ is also continuous. Since $p\circ \wt{f}\circ g=f\circ g=p\circ k$ and $\wt{f}(g(z))=\tx '$, we have $\wt{f}\circ g=k$ by the uniqueness of the function $k$. Thus $g^{-1}(\wt{f}^{-1}(U))=k^{-1}(U)$ is open in $Z$.

The last statement of the lemma follows directly from the fact that $id:c(\tX)\to\tX$ is a disk-covering which induces an isomorphism of fundamental groups.
\end{proof}
\begin{corollary}
Every $\mcc$-covering subgroup of $\pionex$ is also a $\hullc$-covering subgroup of $\pionex$.
\end{corollary}
\begin{example}\label{nonequivdiskcoveringexample}
The previous lemma indicates that if $X$ is path-connected but not an object of $\hullc$, then the identity functions $id:X\to X$ and $\psi(id):c(X)\to X$ are non-isomorphic disk-coverings (only the later of which is a $\hullc$-covering) which correspond to isomorphic $G$-sets under the functor $\mu$. For instance, when $\hullc=\lpcz$, a non-locally path-connected space $X$ provides an example.
\end{example}
\begin{corollary}\label{faithfulfunctorprop}
The coreflection $c:\topz\to\hullc$ induces a fully faithful functor $\psi:\ccovx\to\hccovx$ which is the identity on all underlying sets and functions. Moreover, the following diagram of functors commutes up to natural isomorphism.\[\xymatrix{
\ccovx \ar[dr]_-{\mu} \ar[rr]^-{\psi} && \hccovx \ar[dl]^-{\mu} \\ & G\mathbf{Set}
}\]
\end{corollary}
\begin{proof}
Lemma \ref{ccovtohullccovlemma} shows that a $\mcc$-covering $p:\tX\to X$ gives rise to the $\hullc$-covering $\psi(p):c(\tX)\to X$. If $q:\tY\to X$ is another $\mcc$-covering and $f:\tX\to\tY$ is a map such that $q\circ f=p$, then the coreflection $c(f):c(\tX)\to c(\tY)$ satisfies $\psi(q)\circ c(f)=\psi(p)$. Functorality follows easily from here. Since $c:\ccovx\to\hccovx$ is the identity functor on underlying sets and functions, it is obviously faithful.

The component of the natural isomorphism at a given $\mcc$-covering $p:\tX\to X$, is the isomorphism of $G$-sets $(\psi(p))^{-1}(x_0)\to p^{-1}(x_0)$ induced by the following morphism of disk-coverings over $X$.
\[\xymatrix{
c(\tX) \ar[dr]_-{\psi(p)} \ar[rr]^-{id} && \tX \ar[dl]^-{p} \\ & X
}\]Naturality is straightforward to check. Finally, $\psi$ is full since it is faithful and the monodromy functors are fully faithful.
\end{proof}
\begin{proposition}\label{twoofthreehullccovlemma}
Suppose the objects of $\mcc$ are simply connected and include the unit disk. Let $p:\tX\to X$ and $q:\tY\to \tX$ be maps.
\begin{enumerate}
\item If $p$ and $q$ are $\hullc$-coverings, then so is $p\circ q$,
\item If $p$ and $p\circ q$ are $\hullc$-coverings, then so is $q$,
\item If $q$ and $p\circ q$ are $\hullc$-coverings, then so is $\psi(p):c(\tX)\to X$.
\end{enumerate}
\end{proposition}
\begin{proof}
(1) and (2) follow directly from Proposition \ref{twoofthreeprop}. For (3) suppose $q$ and $r$ are $\hullc$-coverings. By Lemma \ref{ccovtohullccovlemma}, it suffices to show that $p:\tX\to X$ is a weak $\mcc$-covering. Let $Y\in \mcc$ and $f:(Y,y)\to (X,x)$ be a map and $p(\tx)=x$. Since $q$ is surjective, we may pick a point $\ty\in q^{-1}(\tx)$. By Lemma \ref{twoofthreediskcovlemma}, $p$ is a disk-covering and thus has unique lifting with respect to path-connected spaces. Therefore, we only need to check that a based lift of $f$ to $(\tX,\tx)$ exists. Since $Y$ is simply connected, $f_{\ast}(\pi_1(Y,y))=1\subseteq (p\circ q)_{\ast}(\pi_1(\tY,\ty))$. Thus there is a unique map $\wt{f}_{\ty}:(Y,y)\to(\tY,\ty)$ such that $p\circ q\circ \wt{f}_{\ty}=f$. The map $q\circ \wt{f}_{\ty}:(Y,y)\to (\tX,\tx)$ is the desired lift.
\end{proof}
\begin{theorem}\label{reflectionequivlemma1}
For every $X\in\topz$, the continuous identity function $c(X)\to X$ induces an equivalence of categories $\hccovx\cong\hccovcx$
\end{theorem}
\begin{proof}
Define the functor $F:\hccovx\to\hccovcx$ simply by applying the coreflection $c$: If $p:\tX\to X$ is a $\hullc$-covering, then $F(p):\tX=c(\tX)\to c(X)$ is continuous by Proposition \ref{coreflectionbasicprop}. Given a morphism $f:\tX\to\tY$ of $\hullc$-coverings as in the diagram below, we have $\tX,\tY\in\hullc$ so $F$ is the identity on morphisms.
\[\xymatrix{ \tX \ar[rr]^-{f} \ar[dr]_-{p} && \tY \ar[dl]^-{q} & \tX \ar[rr]^-{f} \ar[dr]_{F(p)} && \tY \ar[dl]^-{F(q)}\\
& X & & & c(X) }\]
The inverse $G:\hccovcx\to\hccovx$ is defined as follows: If $p':\tX\to c(X)$ is a $\hullc$-covering, then the composition $p=G(p')=\tX\to c(X)\to X$ is a $\hullc$-covering by Proposition \ref{twoofthreeprop}. On morphisms, $G$ is the identity. A straightforward check shows that $F$ and $G$ are inverse equivalences.
\end{proof}
Theorem \ref{reflectionequivlemma1} implies that if one wishes to consider the $\hullc$-coverings of a space $X$, then for all practical purposes one may assume $X$ is an object of $\hullc$. We now consider subcategories $\mcc\subseteq\mcd$ of $\topz$. The fact that $\hullc\subset\hulld$ allows us to apply the coreflection functor $c:\topz\to\hullc$ to every object of $\hulld$.
\begin{corollary}
If $p:\tX\to X$ is a $\hulld$-covering, then $\psi(p):c(\tX)\to X$ is a $\hullc$-covering.
\end{corollary}
\begin{proof}
Since $\mcc\subset \hulld$, a $\hulld$-covering $p:\tX\to X$ is a weak $\mcc$-covering. Now apply Lemma \ref{ccovtohullccovlemma}.
\end{proof}
The coreflection $c:\hulld\to\hullc$ induces a functor $\phi:\mathbf{Cov}_{\hulld}(X)\to\hccovx$: on objects $\phi$ sends a $\hulld$-covering $p:\tX\to X$ to $\phi(p)=p:c(\tX)\to X$ and if $q:\tY\to X$ is another $\hulld$-covering and $f:\tX\to\tY$ satisfies $q\circ f=p$, then $\phi(f)=c(f):c(\tX)\to c(\tY)$ is the coreflection.
\begin{theorem}\label{embedding1}
If $\mcc\subset\mcd\subset\topz$, the functor $\phi:\mathbf{Cov}_{\hulld}(X)\to\hccovx$ is fully faithful. Moreover, the following diagram of fully faithful functors commutes up to natural isomorphism.\[\xymatrix{
\mathbf{Cov}_{\hulld}(X) \ar[dr]_-{\mu} \ar[rr]^-{\phi} && \hccovx \ar[dl]^-{\mu} \\ & G\mathbf{Set}
}\]
\end{theorem}
\begin{proof}
By definition, $\phi$ is the identity on underlying sets and functions, which makes faithfulness clear. Suppose $p:\tX\to X$ and $q:\tX\to X$ are $\hulld$-coverings and $f:c(\tX)\to c(\tY)$ is map such that $\phi(q)\circ f=\phi(p)$. Since $\hullc\subset\hulld$, we have $c(\tX)=\tX$ and $c(\tY)=Y$. Thus $f:\tX\to\tY$ satisfies $q\circ f=p$ and $c(f)=f$. We conclude that $\phi$ is full.

The component of the natural isomorphism at a given $\hulld$-covering $p:\tX\to X$, is the isomorphism of $G$-sets $(\phi(p))^{-1}(x_0)\to p^{-1}(x_0)$ induced by the morphism $id:c(\tX)\to \tX$ of disk-coverings $\phi(p)$ and $p$ over $X$. Naturality is straightforward to check.
\end{proof}
We conclude that using a smaller category $\mcc$ allows us to retain more of the subgroup lattice of $\pionex$.
\begin{corollary}\label{comparingcovsubgroups}
If $\mcc\subset\mcd\subset\topz$ and $H$ is a $\hulld$-covering subgroup of $\pionex$, then $H$ is also a $\hullc$-covering subgroup of $\pionex$. In particular, if $X$ admits a universal $\hulld$-covering, then $X$ also admits a universal $\hullc$-covering.
\end{corollary}

We summarize the results of this section with a simple diagram: Suppose $D^2\in\mcc\subset\mcd\subset\topz$ with coreflections $\topz\to\hullc$ and $\topz\to\hulld$. Then the following diagram of fully faithful functors commutes up to natural isomorphism (where $\hookrightarrow$ denotes a fully faithful functor).
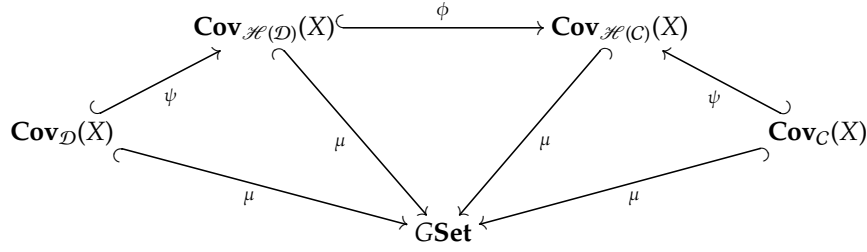
\[\xymatrix{
 & \mathbf{Cov}_{\hulld}(X) \ar@{_{(}->}[ddr]_{\mu} \ar@{^{(}->}[rr]^-{\phi} & & \hccovx \ar@{^{(}->}[ddl]^{\mu}  & \\
\mathbf{Cov}_{\mcd}(X) \ar@{_{(}->}[drr]_{\mu} \ar@{^{(}->}[ur]_{\psi} & & & & \ccovx \ar@{^{(}->}[dll]^{\mu} \ar@{_{(}->}[ul]^-{\psi} \\
&& G\mathbf{Set}
}\]
\subsection{Categorical constructions of $\mcc$-coverings}
The category $\topz$ whose objects are unbased path-connected spaces is not complete (i.e. closed under all small categorical limits). In order to construct limits and pullbacks of coverings, we restrict ourselves to based spaces and maps. Let $\mathbf{bCov}_{\mcc}$ and $\mathbf{bCov}_{\mcc}(X,x)$ denote the categories of all based $\mcc$-coverings $p:(\tX,\tx)\to(X,x)$ and based coverings over $(X,x)$ respectively.
\begin{lemma}\label{completelemma}
The categories $\mathbf{bCov}_{\hullc}$ and $\mathbf{bCov}_{\hullc}(X,x)$ are complete.
\end{lemma}
\begin{proof}
Let $J$ be a small category and $F:J\to \mathbf{bCov}_{\mcc}$ be a diagram of $\hullc$-coverings: $F(j)=p_j:(\tX_j,\tx_j)\to (X_j,x_j)$ for object $j\in J$ and for each morphism $m:j\to k$ in $J$, there are maps $a_m:(\tX_j,\tx_j)\to (\tX_k,\tx_k)$ and $b_m:(X_j,x_j)\to(X_k,x_k)$ such that $p_k\circ a_m=b_m\circ p_j$. Let $\lim  p_j:(\lim \tX_j,\tx_0)\to(\lim X_j,x_0)$ be the standard limit in $\btopz$. Let $X=\lim X_j$, $\tX$ be the path-component of $\lim \tX_j$ containing $\tx_0$ and $p:\tX\to X$ be the restriction of $\lim p_j$. We check that $p:\tX\to X$ is a weak $\hullc$-covering.

For $j\in J$, let $a_j:(\tX,\tx_0)\to (\tX_j,\tx_j)$ and $b_j:(X,x_0)\to (X_j,x_j)$ be the canonical projections satisfying $b_j\circ p=p_j\circ a_j$. Let $f:(Y,y)\to (X,x_0)$ be a map such that $Y\in\hullc$ and $f_{\ast}(\pi_1(Y,y))\subseteq p_{\ast}(\pi_1(\tX,\tx_0))$. Let $f_j=b_j\circ f$ and note that \[(f_j)_{\ast}(\pi_1(Y,y))\subseteq (b_j\circ p)_{\ast}(\pi_1(\tX,\tx_0))=(p_j\circ a_j)_{\ast}(\pi_1(\tX,\tx_0))\subseteq (p_j)_{\ast}(\pi_1(\tX_j,\tx_j))\]Thus there is a unique map $\wt{f}_j:(Y,y)\to (\tX_j,\tx_j)$ such that $p_j\circ \wt{f}_j=f_j$. Given a morphism $m:j\to k$ in $J$, we have
\[p_k\circ a_m\circ \wt{f}_j=b_m\circ p_j\circ \wt{f}_j=b_m\circ f_j=b_m\circ b_j\circ f=b_k\circ f=f_k\]
so $a_m\circ \wt{f}_j=f_k$ by the uniqueness of lifts. It follows that there is a unique map $\wt{f}:(Y,y)\to (\lim \tX_j,\tx_0)$ such that $a_j\circ \wt{f}=\wt{f}_j$. Since $Y$ is path-connected, $\wt{f}$ has image in $\tX$. Finally, since $b_j\circ p\circ \wt{f}=p_j\circ a_j\circ \wt{f}=p_j\circ \wt{f}_j=f_j=b_j\circ f$, the universal property of $X$ gives $p\circ \wt{f}=f$. This proves that $p$ is a weak $\hullc$-covering.
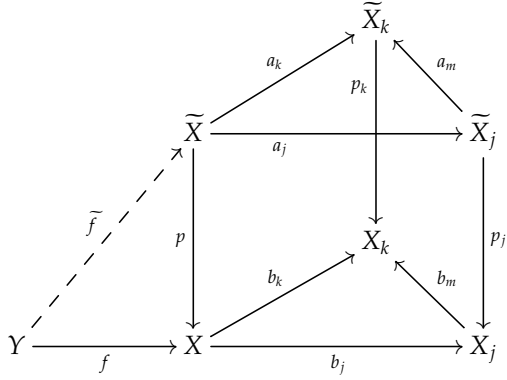
\[\xymatrix{
&&&& \tX_k \ar[dd]_(.3){p_k} \\
&& \tX \ar[dd]_{p} \ar[urr]^{a_k} \ar[rrr]_(.3){a_j} & && \tX_j \ar[ul]_{a_m} \ar[dd]^{p_j}\\
&&&& X_k \\
Y \ar[rr]_{f} \ar@{-->}[uurr]^{\wt{f}} && X  \ar[urr]^{b_k} \ar[rrr]_{b_j} & && X_j \ar[ul]_{b_m}
}\]
By Lemma \ref{ccovtohullccovlemma}, the coreflection $\psi(p):c(\tX)\to X$ is a $\hullc$-covering. It is a standard application of coreflections and limits to check that $p$ is the categorical limit in $\mathbf{bCov}_{\hullc}$.

Applying the same argument to based $\hullc$-coverings over a given based space $(X,x)$, one can prove that $\mathbf{bCov}_{\hullc}(X,x)$ is complete.
\end{proof}
Since $(X,x_0)$ is the limit $\lim (X_j,x_j)$ in the proof of Lemma \ref{completelemma}, there is a canonical homomorphism $h:\pionex\to \lim \pi_1(X_j,x_j)$, $h([\alpha])=([b_j\circ \alpha])$ where the limit $\lim \pi_1(X_j,x_j)$ of groups is realized canonically as a subgroup of the product $\prod_{j\in J}\pi_1(X_j,x_j)$ taken over the objects of $J$. In the following theorem, we utilize this same notation; the result generalizes a key ingredient in the proof that fundamental groups of one-dimensional spaces inject into their first shape group, i.e. are $\pi_1$-shape injective \cite{EK98}.
\begin{theorem}\label{shapeinj}
Suppose the $\hullc$-covering $p:(\tX,\tx_0)\to (X,x_0)$ is the limit a diagram $F:J\to \mathbf{bCov}_{\hullc}$ of universal $\hullc$-coverings. Then $p$ is a universal $\hullc$-covering if and only if the canonical homomorphism $h:\pionex\to \lim \pi_1(X_j,x_j)$ is injective.
\end{theorem}
\begin{proof}
Suppose $h$ is injective. Let $\wt{\alpha}:S^1\to \tX$ be a loop based at $\tx_0$. Since $\tX_j$ is simply connected by assumption, the loop $p_j\circ a_j\circ \wt{\alpha}=b_j\circ p\circ \wt{\alpha}$ is null-homotopic in $X_j$. Since $h([p\circ \wt{\alpha}])=([b_j\circ p\circ \wt{\alpha}])$ is trivial and $h$ is injective, $[p\circ \wt{\alpha}]=1$. Since $p_{\ast}$ is injective, $[\wt{\alpha}]=1$ proving that $\tX$ is simply connected. For the converse, suppose $\alpha:S^1\to X$ is a loop such that $\alpha_j=b_j\circ \alpha$ is null-homotopic in $X_j$ for all $j\in J$. To prove the injectivity of $h$ it suffices to show $\alpha$ is null-homotopic. Since $p_j:\tX_j\to X_j$ is a $\hullc$-covering, there is a unique lift $\wt{\alpha}_j:S^1\to \tX_j$ such that $p_j\circ \wt{\alpha}_j=\alpha_j$. Given a morphism $m:j\to k$ in $J$, we have $p_k\circ a_m\circ \wt{\alpha}_j=b_m\circ p_j\circ \wt{\alpha}_j=b_m\circ \alpha_j=\alpha_k$, which by unique lifting proves that $a_m\circ \wt{\alpha}_j=\wt{\alpha}_k$. Since together, the loops $\wt{\alpha}_j$ and $\alpha_j$ define a cone from the identity $\hullc$-covering $id:S^1\to S^1$ to $F$, we see that there is a loop $\wt{\alpha}:S^1\to \tX$ such that $p\circ \wt{\alpha}=\alpha$ and $a_j\circ \wt{\alpha}=\wt{\alpha}_j$. But since $\tX$ is simply connected, both $\wt{\alpha}$ and $\alpha$ are null-homotopic.
\end{proof}
\begin{corollary}\label{shapeinjcor}
Suppose $(X,x_0)=\lim (X_j,x_j)$ is the limit of a diagram $F:J\to \mathbf{bTop_{0}}$ in the category of based, path-connected spaces where $X_j$ admits a universal $\hullc$-covering for each $j\in J$ and $b_j:X\to X_j$ is the projection. If the canonical homomorphism $h:\pionex\to \lim \pi_1(X_j,x_j)$, $h([\alpha])=([b_j\circ \alpha])$ is injective, then $X$ admits a universal $\hullc$-covering.
\end{corollary}
\begin{proof}
Let $p_j:(\tX_j,\tx_j)\to (X_j,x_j)$ be a universal $\hullc$-covering. Given a morphism $m:j\to k$ in $J$ let $b_m=F(m):X_j\to X_k$. Since $\tX_j$ is simply connected, there is a unique map $a_m:(\tX_j,\tx_j)\to (\tX_k,\tx_k)$ such that $p_k\circ a_m=b_m\circ p_j$. Thus we have a diagram $J\to \mathbf{bCov}_{\hullc}$ of universal $\hullc$-coverings with limit $p:(\tX,\tx_0)\to (X,x_0)$. By Theorem \ref{shapeinj}, $\tX$ is simply connected.
\end{proof}
Corollary \ref{shapeinjcor} provides a categorical proof of the fact that a $\pi_1$-shape injective Peano continuum admits a universal $\lpcz$-covering. The general case appears in \cite{FZ07}. Interestingly, the converse of Corollary \ref{shapeinjcor} is false in the locally-path connected case \cite{FRVZ11}; there exist Peano continua $(X,x_0)= \varprojlim_{n} (X_n,x_n)$, which are the inverse limit of finite polyhedra and which admit a universal $\lpcz$-covering but for which $h:\pionex\to \varprojlim_{n} \pi_1(X_n,x_n)$ fails to be injective.

Another useful construction is the pullback construction: Fix a $\hullc$-covering $p:\tX\to X$ and any map $f:Y\to X$. We view the pullback $\tX\times_{X}Y=\{(\tx,y)\in\tX\times Y|p(\tx)=f(y)\}$ as a subspace of the direct product. Similar to the situation above, the space $\tX\times_{X}Y$ need not path-connected and the components need to be objects of $\hullc$. This failure has been fully characterized in the locally path-connected case $\hullc=\lpcz$ \cite{Fischerpullback}. To overcome this issue, we again choose a single path component and apply the coreflection $c$.

Fix points $y_0\in Y$, $x_0=f(x_0)$, and $\tx_0\in p^{-1}(x_0)$. Let $P$ be the path component of $\tX\times_{X}Y$ containing $(\tx_0,y_0)$ and let $f^{\#}\tX=c(P)$ be the coreflection. We call the projection $f^{\#}p:f^{\#}\tX\to Y$, $f^{\#}p(\tx,y)=y$ the \textit{pullback of }$\tX$ \textit{by} $f$ and check that it is a $\hullc$-covering map.

\begin{lemma}\label{pullbacklemma}
If $p:\tX\to X$ is a $\hullc$-covering, $Y\in\topz$ and $f:Y\to X$ is a map, then the pullback $f^{\#}p:f^{\#}\tX\to Y$ is a $\hullc$-covering corresponding to the subgroup $f_{\ast}^{-1}(p_{\ast}(\pi_1(\tX,\tx_0)))\subseteq \pi_1(Y,y_0)$.
\end{lemma}
\begin{proof}
By construction, $f^{\#}\tX$ is an object of $\hullc$. Let $Z\in\hullc$, $(\tx,y)\in f^{\#}\tX$ and $g:(Z,z)\to(Y,y)$ be a map such that $g_{\ast}(\pi_1(Z,z))\subseteq (f^{\#}p)_{\ast}(\pi_1(f^{\#}\tX,(\tx,y)))$. Let $q:f^{\#}\tX\to\tX$ be the second projection such that $p\circ q=f\circ f^{\#}p$. Now
\begin{eqnarray*}
(f\circ g)_{\ast}(\pi_1(Z,z)) &\subseteq & f_{\ast}\left((f^{\#}p)_{\ast}(\pi_1(f^{\#}\tX,(\tx,y)))\right)\\
&=& p_{\ast}(q_{\ast}(\pi_1(f^{\#}\tX,(\tx,y))))\\
&\subseteq & p_{\ast}(\pi_1(\tX,\tx))
\end{eqnarray*}
Thus there is a unique map $k:(Z,z)\to (\tX,\tx)$ such that $p\circ k=f\circ g$. By the universal property of pullbacks, we have a unique map $\wt{g}:(Z,z)\to (\tX\times_{X} Y,(\tx,y))$ satisfying $q\circ \wt{g}=k$ and $f^{\#}p\circ \wt{g}=g$. Since $Z$ is path-connected $\wt{g}$ has image in $P$ and since $Z\in\hullc$, $\wt{g}:Z\to f^{\#}\tX$ is continuous.
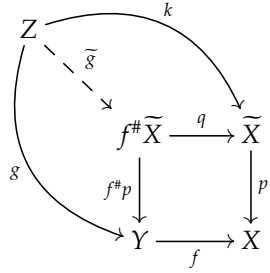
\[\xymatrix{
Z \ar@{-->}[dr]^-{\wt{g}} \ar@/^2pc/[drr]^-{k} \ar@/_2pc/[ddr]_-{g} \\
& f^{\#}\tX  \ar[r]^-{q} \ar[d]_-{f^{\#}p} & \tX  \ar[d]^-{p} \\
& Y \ar[r]_-{f} & X }\]
Finally, we check that $(f^{\#}p)_{\ast}(\pi_1(f^{\#}\tX,(\tx_0,y_0)))=f_{\ast}^{-1}(p_{\ast}(\pi_1(\tX,\tx_0)))$. One inclusion is clear from the commutativity of the square. For the other suppose $[\gamma]\in f_{\ast}^{-1}(p_{\ast}(\pi_1(\tX,\tx_0)))$. Since $[f\circ \gamma]\in p_{\ast}(\pi_1(\tX,\tx_0)$, there is a unique lift $\kappa:S^1\to \tX$ such that $p\circ \kappa=f\circ \gamma$. Thus there is a unique loop $\wt{\gamma}:S^1\to f^{\#}\tX$ such that $f^{\#}p\circ \wt{\gamma}=\gamma$ and $q\circ \wt{\gamma}=\kappa$. It follows that $[\gamma]\in (f^{\#}p)_{\ast}(\pi_1(f^{\#}\tX,(\tx_0,y_0)))$.
\end{proof}
In fact, a based map $f:(Y,y_0)\to (X,x_0)$ induces a functor $f^{\#}:\mathbf{bCov}_{\hullc}(X,x_0)\to\mathbf{bCov}_{\hullc}(X,y_0)$. If $p':(\tX ',\tx_{0} ')\to (X,x_0)$ is another $\hullc$-covering over $X$ and $g:(\tX,\tx_0)\to (\tX ',\tx_{0} ')$ is a morphism such that $p'\circ g=p$, then $f^{\#}g:f^{\#}\tX\to f^{\#}\tX '$ is uniquely induced by a straightforward pullback diagram.
%
%
\begin{corollary}\label{simplyconnectedpullback}
Suppose $p:\tX\to X$ is a universal $\hullc$-covering. Then the pullback $f^{\#}p:f^{\#}\tX\to Y$ is a universal $\hullc$-covering if and only if $f_{\ast}:\pi_1(Y,y_0)\to \pi_1(X,x_0)$ is injective.
\end{corollary}
%
%
%
We use the existence of pullbacks to verify the closure of $\hullc$-covering subgroups under intersection.
\begin{theorem}\label{intersectiontheorem}
If $\{H_j|j\in J\}$ is any set of $\hullc$-covering subgroups of $\pionex$, then $\bigcap_{j}H_j$ is a $\hullc$-covering subgroup.
\end{theorem}
\begin{proof}
Fix a $\hullc$-covering $p_j:(\tX_j,\tx_j)\to (X,x_0)$ such that $(p_j)_{\ast}(\pi_1(\tX_j,\tx_j))=H_j$. Let $(\tX,\tx_0)=(\prod_{j}\tX_j,(\tx_j))$. By Lemma \ref{completelemma}, the product $p=\prod p_j:c(\tX)\to \prod_jX$ is a $\hullc$-covering. Note that $p_{\ast}(\pi_1(\tX,\tx_0))$ is the product of subgroups $\prod_{j}H_j$ when we make the identification $\prod_{j}\pionex\cong\pi_1(\prod_{j}X,(x_0))$. Let $\delta:X\to\prod_{j}X$ be the diagonal map and $\delta^{\#}p:\delta^{\#}\tX\to X$ be the pullback of $\tX$ by $\delta$. By Lemma \ref{pullbacklemma}, $\delta^{\#}p$ is a $\hullc$-covering map such that the image of the homomorphism
$(\delta^{\#}p)_{\ast}$ is $\delta_{\ast}^{-1}(\prod_{j}H_j)$. It is easy to see that $\delta_{\ast}^{-1}(\prod_{j}H_j)=\bigcap_{j}H_j$ and thus $\bigcap_{j}H_j$ is a $\hullc$-covering subgroup.
%
\end{proof}
Theorem \ref{intersectiontheorem} implies the existence of a minimal $\hullc$-covering subgroup for every based path-connected space $(X,x_0)$: Let \[\txt{$ U(X,\mcc)=\bigcap\{H|H$ is a $\hullc$-covering subgroup of $\pionex\}$.}\] The existence of minimal $\hullc$-covering subgroups immediately implies the existence of an initial based $\hullc$-covering $p:(\tX,\tx_0)\to (X,x_0)$ such that $p_{\ast}(\pi_1(\tX,\tx_0))=U(X,\mcc)$.
\begin{theorem}
For any space $X$, the category $\hccovx$ has an initial object.
\end{theorem}
\begin{corollary}
A space $X$ admits a universal $\hullc$-covering if and only if $U(X,\mcc)=1$.
\end{corollary}
Any non-trivial elements of the subgroup $U(X,\mcc)$ may be viewed as those elements of $\pionex$ which are indistinguishable from the homotopy class of the constant loop with respect to $\hullc$-coverings. Thus if $U(X,\mcc)\neq 1$ there are homotopy classes of loops which are inaccessible to study via $\hullc$-coverings.
\begin{proposition}
$U(X,\mcc)$ is a normal subgroup of $\pionex$.
\end{proposition}
\begin{proof}
Recall from Remark \ref{conjugateremark} that if $H$ is a $\hullc$-covering subgroup of $G=\pionex$, then $gHg^{-1}$ is also a $\hullc$-covering subgroup for each $g\in G$. The normal subgroup $N_H=\bigcap_{g\in G}gHg^{-1}$ (called the \textit{core} of $H$ in $G$) is the largest normal subgroup of $G$ which is contained in $H$. By Theorem \ref{intersectiontheorem}, $N_H$ is a $\hullc$-covering subgroup. Let $N$ be the intersection of the cores of all $\hullc$-covering subgroups $H$. Clearly $N$ is normal. Since $N_H\subseteq H$ for all $H$, we have $N\subseteq U(X,\mcc)$ and since $N$ is a $\hullc$-covering subgroup (by Theorem \ref{intersectiontheorem}), we see that $N= U(X,\mcc)$.
\end{proof}
\section{The groupoid approach}
Recall that if a map $f:X\to Y$ has the unique path-lifting property and $f\circ \alpha=f\circ \beta$ for paths $\alpha$ and $\beta$, then we cannot conclude that $\alpha=\beta$ unless we already know they agree at a point. Consequently, the basepoints in the definition of $\mcc$-covering cannot be disregarded. Nevertheless, some of the above results have nice analogues involving the fundamental groupoid, which avoid reference to subgroup conjugacy classes. We briefly mention a few of them here.

Given a small groupoid $\mathcal{G}$, let $\mathcal{G}(x,-)$ denote the set of morphisms in $\mathcal{G}$ whose source is the object $x$. A \textit{covering morphism} over $\mathcal{G}$ is a functor $F:\mathcal{H}\to \mathcal{G}$ between groupoids such that for each object $y\in \mathcal{H}$ the induced function $\mathcal{H}(y,-)\to \mathcal{G}(F(y),-)$ is bijective \cite{Brown06}. The category of covering morphisms over $\mathcal{G}$ is denoted $\textbf{CovMor}(\mathcal{G})$. It is a standard fact that $\textbf{CovMor}(\mathcal{G})$ is equivalent to the functor category $\mathbf{Set}^{\mathcal{G}}$ of operations of $\mathcal{G}$ on sets. Moreover, if $G=\mathcal{G}(x,x)$ is the vertex group viewed as a one-object subcategory and $\mathcal{G}$ is connected, then the inclusion $G\to\mathcal{G}$ is an equivalence of categories which in turn induces an equivalence $\mathbf{Set}^{\mathcal{G}}\to G\mathbf{Set}$. We are interested in the case where $\mathcal{G}=\pi_1(X)$ is the fundamental groupoid of $X$ and $G=\pi_1(X,x_0)$.

Since a disk-covering $p:E\to X$ has unique lifting of all paths and homotopies of paths, it is clear that the induced functor $\pi_{1}(p):\pi_1(E)\to \pi_1(X)$ is a covering morphism. Thus we obtain a monodromy functor $M:\dcovx\to \textbf{CovMor}(\pi_1(X))$ by applying $\pi_1$ to disk-coverings and their morphisms. If $E:\textbf{CovMor}(\pi_1(X))\to \pionex\mathbf{Set}$ is the equivalence of categories referenced in the previous paragraph, then $E\circ M\simeq\mu$ is the faithful monodromy functor of Lemma \ref{faithfulpropdiskcov}. Consequently, $M$ is faithful.

For the rest of this section, suppose $\mcc\subset \topz$ is a category containing the unit disk as an object.
\begin{theorem}
For any space $X$, there is a canonical fully faithful monodromy functor $M:\ccovx \to \textbf{CovMor}(\pi_1(X))$.
\end{theorem}
\begin{proof}
Since $\ccovx\subset \dcovx$, we define the functor $M:\ccovx\to \textbf{CovMor}(\pi_1(X))$ simply to be the restriction of $M:\dcovx\to \textbf{CovMor}(\pi_1(X))$. It follows from Theorem \ref{fullyfaithfulthm} that $M:\ccovx\to \textbf{CovMor}(\pi_1(X))$ is fully faithful.
\end{proof}
Since the following diagram commutes up to natural isomorphism and the bottom functor is an equivalence, it is possible to consistently replace $\mu$ with $M$ in all of the previous diagrams involving monodromy.
\[\xymatrix{
& \ccovx \ar@{^{(}->}[dl]_-{M} \ar@{_{(}->}[dr]^-{\mu} \\ \textbf{CovMor}(\pi_1(X)) \ar[rr]^-{\cong}_-{E} &&  \pionex\mathbf{Set}}\]

The results on categorical limits and pullbacks as applied to based $\hullc$-coverings may be summarized as follows.

\begin{theorem}
If $\textbf{CovMor}$ is the category of all covering morphisms, then for any $\mcc$ the image of $M:\mathbf{bCov}_{\hullc}\to \textbf{CovMor}$ is complete.
\end{theorem}
\begin{theorem} For any $(X,x_0)\in \btopz$, the image of $M:\mathbf{bCov}_{\hullc}(X,x_0)\to \textbf{CovMor}(\pi_1(X))$ is complete and has an initial object.
\end{theorem}
\section{The universal covering theory: $\Delta$-coverings}\label{deltacoveringsection}
In this section, we consider in more detail the category in which all other $\mcc$-covering categories embed, that is, when $D^2$ is the only object of $\mcc$.
\begin{definition}
A subset $U\subseteq X$ is $\Delta$\textit{-open} in $X$ if $f^{-1}(U)$ is open in $D^2$ for every map $f:D^2\to X$. A space $X$ is $\Delta$\textit{-generated} if a set $U$ is open in $X$ if and only if $U$ is $\Delta$-open in $X$.
\end{definition}
Let $\Delta\topz$ denote the category of path-connected, $\Delta$-generated spaces. We prefer to use the following characterization of $\Delta$-generated topologies which follows directly from the existence of space filling curves $\ui\to D^2$: A set $U\subseteq X$ is $\Delta$-open if and only if $\alpha^{-1}(U)$ is open in $\ui$ for every path $\alpha:\ui\to X$.

Recall a space $X$ is \textit{sequential} if $U\subset X$ is open if and only if for every convergent sequence $x_n\to x$ in $X$ with $x\in U$, there exists $N\geq 1$ such that $x_n\in U$ for all $n\geq N$. The following lemma follows from well-known topological facts.
\begin{lemma}\label{deltagenpropslemma}\cite{CSWdtopology}
Every first countable and locally path-connected space is $\Delta$-generated and every $\Delta$-generated space is sequential and locally path-connected.
\end{lemma}
If $\mcd$ denotes the full subcategory of $\topz$ whose only object is $D^2$, then $\scrh(\mcd)=\Delta\topz$. We denote the coreflection functor by $\Delta:\topz\to\Delta\topz$. Thus $\Delta(X)$ has the same underlying set as $X$ and has the topology consisting of all $\Delta$-open sets. For the sake of convenience we call a $\Delta\topz$-covering $p:\tX\to X$ simply a $\Delta$\textit{-covering}. Definition \ref{ccovdef} translates as follows.
\begin{definition}\label{deltacovdef}
A map $p:\tX\to X$ is a $\Delta$\textit{-covering map} if
\begin{enumerate}
\item $\tX\in\Delta\topz$,
\item for every space $Y\in\Delta\topz$, point $\tx\in\tX$, and based map $f:(Y,y)\to (X,p(\tx))$ such that $f_{\ast}(\pi_1(Y,y))\subset p_{\ast}(\pi_1(\tX,\tx))$, there is a unique map $\wt{f}:(Y,y)\to (\tX,\tx)$ such that $p\circ \wt{f}=f$.
\end{enumerate}
\end{definition}
The following theorem is the relevant combination of Theorem \ref{embedding1} and Corollary \ref{faithfulfunctorprop}; it implies that among all $\mcc$-coverings, $\Delta$-coverings retain the most information about the subgroup structure of $\pionex$.
\begin{theorem}
If $\mcc$ has $D^2$ as an object, then there is a canonical fully faithful functor $\ccovx\to\delcovx$.
\end{theorem}
\begin{corollary}
Suppose $\mcc\subset\topz$ has $D^2$ as an object. Then every $\mcc$-covering subgroup of $\pionex $ is also a $\Delta$-covering subgroup of $\pionex$.
\end{corollary}
We now observe that the images of the functors $\mu:\delcovx\to G\mathbf{Set}$ and $\mu:\mathbf{DCov}(X)\to G\mathbf{Set}$ agree, or equivalently that every disk-covering subgroup of $\pionex$ is a $\Delta$-covering subgroup of $\pionex$. Recall that a disk-covering is precisely a weak $\mcd$-covering where the only object of $\mcd$ is $D^2$. The following lemma is a special case of Lemma \ref{ccovtohullccovlemma}.
\begin{lemma}
If $X$ is path-connected and $p:E\to X$ is a disk-covering, then $p:\Delta(E)\to X$ is a $\Delta$-covering.
\end{lemma}
We can now extend the coreflection $\Delta:\topz\to\Delta\topz$ to a functor $\dcovx\to\delcovx$. If $f:E\to E'$ satisfies $p'\circ f=p$ for disk-coverings $p:E\to X$ and $p:E'\to X$, then $\Delta(f):\Delta(E)\to \Delta(E')$ is a morphism of the $\Delta$-coverings $p:\Delta(E)\to X$, $p:\Delta(E')\to X$.
\begin{corollary}
The functor $\Delta:\dcovx\to\delcovx$ taking disk-covering $p:E\to X$ to the $\Delta$-covering $p:\Delta(E)\to X$ is right adjoint to the inclusion $\delcovx\to\dcovx$.
\end{corollary}
\begin{theorem}\label{naturaliso1}
Suppose $X$ is path-connected and $G=\pionex$. The follow diagram of functors commutes up to natural isomorphism.
\[\xymatrix{
\dcovx \ar[dr]_-{\mu} \ar[rr]^-{\Delta} &&\delcovx  \ar@{^{(}->}[dl]^-{\mu}\\
& G\mathbf{Set}
}\]
\end{theorem}
\begin{proof}
To define the natural isomorphism $\eta:\mu\circ\Delta\to \mu$, we take the component $\eta_p$ of the disk-covering $p:E\to X$ to be the isomorphism of group actions induced by the morphism \[\xymatrix{\Delta(E) \ar[dr]_{p} \ar[rr]^-{id} && E \ar[dl]^{p}\\ & X }\]of disk-coverings under $\mu$. The naturality of $\eta$ is straightforward to verify.
\end{proof}
\begin{remark}\label{weakremark}
We do not wish to give much attention to categories of weak $\mcc$-coverings; however, Theorem \ref{naturaliso1} can be generalized as follows: If $\mathbf{wCov}_{\mcc}(X)$ is the category of weak $\mcc$-coverings over $X$, then the follow diagram of functors commutes up to natural isomorphism where $\psi$ is right adjoint to the inclusion $\hccovx\subset\mathbf{wCov}_{\mcc}(X)$.
\[\xymatrix{
\mathbf{wCov}_{\mcc}(X)\ar[dr]_-{\mu} \ar[rr]^-{\psi} && \hccovx  \ar@{^{(}->}[dl]^-{\mu}\\
& G\mathbf{Set}
}\]
\end{remark}
One should interpret the conclusion of Theorem \ref{naturaliso1} as follows: Any attempt to study the combinatorial structure of the traditional fundamental group $\pionex$ using maps that uniquely lift paths and homotopies of paths can be replaced or generalized by $\Delta$-coverings without losing any information about $\pionex$. Additionally, by Theorem \ref{reflectionequivlemma1}, we have $\delcovx\cong \Delta\mathbf{Cov}(\Delta(X))$ and thus there is also no loss of information by assuming from the start that $X$ is $\Delta$-generated. The subgroup $U(X,\Delta\topz)\subseteq \pionex$ is, in some sense, the group of ``phantom" homotopy class, i.e. those which are invisible to any generalized covering-theoretic approach.

Given a $\Delta$-covering $p:\tX\to X$, it is natural to search for explicit descriptions of the topology of $\tX$. In the following section, we employ the locally path-connected category to provide a simple description of $\tX$ in the case that $X$ is first countable. For a non-first countable, $\Delta$-generated space $X$, it seems that we can only describe the topology of $\tX$ as a final topology with respect to a collection of paths $\ui\to \tX$.
\section{Coverings for locally path-connected spaces: $\lpcz$-coverings}
The second case we consider is when $\hullc$ is the category $\mathbf{lpc}_{0}$ of path-connected, locally path-connected spaces. As discussed in the introduction, the notion of $\mathbf{lpc}_{0}$-covering is a direct generalization of the ``generalized regular coverings" in \cite{FZ07}. $\lpcz$-coverings $p$, which are non-regular in the sense that the image of $p_{\ast}$ is a non-normal subgroup, are known to exist \cite{FZ13}.
\begin{definition}\cite{Dydak11}
Let $(J,\leq )$ be an infinite directed set. The \textit{directed wedge} $(A,a_0)=\bigvee_{J}(X_j,x_j)$ of spaces $(X_j,x_j)$ indexed by $J$ is the based wedge sum with the following topology: $U\subset A\backslash \{a_0\}$ is open if and only if $U\cap X_j$ is open in $X_j$ for each $j\in J$ and $U$ is an open neighborhood of $a_0$ if and only if there is a $k\in J$ such that $X_j\subset U$ for all $j> k$ and $U\cap X_j$ is open in $X_j$ for each $j\in J$. The \textit{arc-hedgehog over} $J$ is the directed wedge $(H(J),a_0)=\bigvee_{J}(\ui_j,0)$.
\end{definition}
If $t\in \ui$, we let $t_j$ denote the image of $t\in \ui_j$ in $H(J)$. Observe that maps $(H(J),a_0)\to (X,x)$ are in bijective correspondence with convergent nets $\alpha_j\to c_{x}$, $j\in J$ (which converge to the constant path) in the path space $P(X)$. Also note that $H(J)$ is path-connected, locally path-connected, and simply-connected. Let $\mathcal{H}$ be the category whose objects include $H(J)$ for every directed set $J$.
\begin{lemma}\label{hedgehogislpc}
$\mathscr{H}(\mathcal{H})= \lpcz$.
\end{lemma}
\begin{proof}
Since every quotient of a topological sum of locally path-connected spaces is locally path-connected, it is clear that $\mathscr{H}(\mathcal{H})\subseteq \lpcz$. Suppose $X$ is path-connected and locally path-connected and suppose there is a set $U\subseteq X$ such that $\epsilon^{-1}(U)$ is open for every map $\epsilon:H(J)\to X$ from some arc-hedgehog. We check that $U$ is open in $X$. If $U$ is not open, then there exists $x\in U$ such that no neighborhood of $x$ is contained in $U$. Let $\{V_j|j\in J\}$ be a base of path-connected connected neighborhoods at $x$, where $J$ is a directed set. For every $j\in J$, there is a point $x_j\in V_j\backslash U$. Pick a path $\epsilon_j:\ui\to V_j$ from $x$ to $x_j$. Together these paths induce a unique map $\epsilon:(H(J),a_0)\to (X,x)$ such that the restriction to $\ui_j$ is $\epsilon_j$. Since $\epsilon$ is continuous $\epsilon^{-1}(U)$ is an open neighborhood of $a_0$ in $H(J)$. Thus there is a $k\in J$ such that $\ui_j\subseteq \epsilon^{-1}(U)$ for all $j>k$. Thus $\epsilon_j(1)=x_j\in U$ for all $j>k$ which is a contradiction. Thus $U$ must be open in $X$.
\end{proof}
The coreflection functor $\topz\to\mathscr{H}(\mathcal{H})$ is well-known \cite{BDLM08,Dydak11,FZ07}; we denote it by $lpc:\topz\to\lpcz$. A convenient description of the coreflection is the following: $lpc(X)$ has the same underlying set as $X$ and has topology generated by the path-components of the open sets of $X$. Since $D^2$ is not an object of $\mathcal{H}$, apparently, we cannot apply much of the above theory to $\mathcal{H}$ itself. However, it is certainly true that $D^2\in \lpcz$. Since $\scrh(\lpcz)=\lpcz$, the results in Section 2 apply to $\lpcz$ when we view it as the coreflective hull of $\mathcal{H}$ with $D^2$ added as an object. Definition \ref{ccovdef} translates as follows.
\begin{definition}\label{lpccovdef}
A map $p:\tX\to X$ is a $\lpcz$\textit{-covering map} if
\begin{enumerate}
\item $\tX$ is path-connected and locally path-connected,
\item for every path-connected, locally path-connected space $Y$, point $\tx\in\tX$, and based map $f:(Y,y)\to (X,p(\tx))$ such that $f_{\ast}(\pi_1(Y,y))\subset p_{\ast}(\pi_1(\tX,\tx))$, there is a unique map $\wt{f}:(Y,y)\to (\tX,\tx)$ such that $p\circ \wt{f}=f$.
\end{enumerate}
\end{definition}
\begin{example}
Suppose $X$ is semilocally simply connected and locally path-connected. Certainly every covering (in the classical sense) over $X$ is a $\lpcz$-covering. Thus the traditional classification of covering spaces guarantees that every subgroup of $X$ is a $\lpcz$-covering subgroup. As a consequence, the category $\mathbf{Cov}_{\lpcz}(X)$ of $\lpcz$-coverings is precisely the usual category of connected coverings over $X$.
\end{example}
\begin{example}
It is known that if $X$ is a metric space with Lebesgue dimension equal to $1$ or if $X$ embeds in $\mathbb{R}^2$, then $X$ admits a universal $\lpcz$-covering, or equivalently $U(X,\lpcz)=1$ \cite{FZ07}. Thus spaces such as the Hawaiian earring, Menger curve, and Sierpinski triangle all admit a universal $\lpcz$-covering which, for many practical purposes, is a suitable replacement for a traditional universal covering.
\end{example}
We now seek to identify the topology of $\lpcz$-covering spaces. The authors of \cite{FZ07} show that the so-called ``whisker topology" construction, used in classical covering theory \cite{Spanier66}, provides a candidate one might use to construct $\lpcz$-coverings. We confirm that all $\lpcz$-coverings arise from the whisker topology.

Fix $(X,x_0)\in\btopz$ and a subgroup $H\subseteq\pionex$. Let $\tXh=\pxxo/\sim$ where $\alpha\sim \beta$ if and only if $\alpha(1)=\beta(1)$ and $[\alpha\cdot\beta^{-}]\in H$. The equivalence class of $\alpha$ is denoted $[\alpha]_H$. We give $\tXh$ the \textit{whisker topology}, which is generated by the sets $$B([\alpha]_H,U)=\left\{[\alpha\cdot\epsilon]_H|\epsilon([0,1])\subseteq U\right\}$$ where $U$ is an open neighborhood of $\alpha(1)$. Note that if $[\beta]_H\in B([\alpha]_H,U)$, then $B([\alpha]_H,U)=B([\beta]_H,U)$. In the case that $H=1$ is the trivial subgroup, we just write $\tX$. The basepoint of $\tXh$ is taken to be class $\txh=[c_{x_0}]_H$ of the constant path $c_{x_0}$ at $x_0$. The map $p_H:\tXh\to X$ is defined to be the endpoint projection $p_H([\alpha]_H)=\alpha(1)$.

When $X$ is locally path-connected and semilocally simply-connected, $p_H:\tXh\to X$, $p_H(\txh)=x_0$ is the classical covering space of $X$ which corresponds to the subgroup $H\subseteq \pionex$ under the usual classification of coverings \cite{Spanier66}. When $X$ is not semilocally simply-connected, $p_H:\tXh\to X$ may not be a covering map but still has the potential to be a $\lpcz$-covering map. The general properties of $p_H$ are studied in great detail in \cite{FZ07}.

\begin{lemma}\cite{FZ07}
The space $\tXh$ is path-connected and locally path-connected. The endpoint projection $p_H:\tXh\to X$ is a continuous surjection. If $X$ is locally path-connected, then $p_H:\tXh\to X$, $p_H([\alpha]_H)=\alpha(1)$ is an open map.
\end{lemma}

An important feature of the whisker topology is that many maps $f:(Y,y)\to (X,x)$ are guaranteed to have at least one lift to $\tXh$. The following result is extracted from classical covering space theory.

\begin{lemma}\cite[Lemma 2.4]{FZ07}\label{liftsexist} Suppose $Y$ is path-connected and locally path-connected, $y_0\in Y$, $\alpha\in P(X,x_0)$, and $f:(Y,y_0)\to (X,\alpha(1))$ is a map such that $f_{\ast}(\pi_1(Y,y_0))\subseteq [\alpha^{-}]H[\alpha]$. Then there is a continuous map $\wt{f}:(Y,y_0)\to (\tXh,[\alpha]_H)$ such that $p_H\circ \wt{f}=f$ defined by $\wt{f}(y)=[\alpha\cdot (f\circ \tau)]_H$ where $\tau:\ui\to Y$ is any path from $y_0$ to $y$.
\end{lemma}
%
%
%
The lift $\wt{f}:(Y,y_0)\to(\tXh,\txh)$ of a map $f:(Y,y_0)\to(X,x_0)$ defined in Lemma \ref{liftsexist} is called the \textit{standard lift of }$f$. In particular, every path $\alpha:(\ui,0)\to (X,x_0)$ has a standard lift $\wt{\alpha}_{\mathscr{S}}:\ui\to\tXh$ starting at $\txh$ defined as $\wt{\alpha}_{\mathscr{S}}(t)=[\alpha_t]_H$ where $\alpha_t(s)=\alpha(st)$.
\begin{corollary} \label{hinclusion}
$H\subseteq (p_H)_{\ast}\left(\pi_1(\tXh,\txh)\right)$.
\end{corollary}
\begin{proof}
If $[\alpha]\in H$, then the standard lift $\wt{\alpha}_{\mathscr{S}}$ is a loop based at $\txh$ since $\wt{\alpha}_{\mathscr{S}}(1)=[\alpha_1]_H=[\alpha]_H=\txh$. Thus $(p_H)_{\ast}\left(\left[\wt{\alpha}_{\mathscr{S}}\right]\right)=[\alpha]$.
\end{proof}
\begin{lemma}\label{genunivcov}
For any subgroup $H\subseteq\pionex$, the following are equivalent:
\begin{enumerate}
\item $p_H:\tXh\to X$ is a $\lpcz$-covering,
\item $p_H:\tXh\to X$ has the unique path lifting property,
\item $(p_H)_{\ast}\left(\pi_1(\tXh,\txh)\right)=H$.
\end{enumerate}
\end{lemma}
\begin{proof}
(1) $\Rightarrow$ (2) is obvious. (2) $\Rightarrow$ (3) Suppose $p_H:\tXh\to X$ has the unique path lifting property. By Corollary \ref{hinclusion}, one inclusion is clear. Suppose $\gamma:\ui\to\tXh$ is a loop based at $\txh$. Let $\alpha=p_H\circ \gamma$. By uniqueness of path lifting, $\gamma=\wt{\alpha}_{\mathscr{S}}$. In particular, $\wt{\alpha}_{\mathscr{S}}(1)=[\alpha]_H=\txh$, which implies $(p_H)_{\ast}([\gamma])=[\alpha]\in H$.

(3) $\Rightarrow$ (2) Suppose $(p_H)_{\ast}\left(\pi_1(\tXh,\txh)\right)=H$. Suppose $\alpha\in P(X,x_0)$ is a path such that there is a lift $\beta:\ui\to\tXh$, $\beta(0)$ such that $\beta(1)\neq \wt{\alpha}_{\mathscr{S}}(1)=[\alpha]_H$. Let $\beta(t)=[\beta_t]_H$ and $\gamma=\beta_1$ and note $\beta\cdot\left(\wt{\gamma}_{\mathscr{S}}\right)^{-}$ is a well-defined loop based at $\txh$. By assumption, $[\alpha\cdot\beta_{1}^{-}] =(p_H)_{\ast}\left(\left[\beta\cdot\left(\wt{\gamma}_{\mathscr{S}}\right)^{-}\right]\right)\in H$. Thus $\beta(1)=[\beta_1]_{H}=[\alpha]_H=\alpha(1)$ which is a contradiction.

(2) \& (3) $\Rightarrow$ (1). Let $Y\in\lpcz$, $y\in Y$, $\alpha\in P(X,x_0)$, and $f:(Y,y)\to (X,x)$ such that $\alpha(1)=x$ and $f_{\ast}(\pi_1(Y,y))\subseteq (p_H)_{\ast}(\pi_1(\tXh,[\alpha]_H))$. We have $$(p_H)_{\ast}(\pi_1(\tXh,[\alpha]_H)) =[\alpha^{-}](p_H)_{\ast}(\pi_1(\tXh,\txh))[\alpha]=[\alpha^{-}]H[\alpha].$$
By Lemma \ref{liftsexist}, there is a map $\wt{f}:(Y,y)\to(\tXh,[\alpha]_H)$ such that $p_H\circ \wt{f}=f$. Since $p_H$ has the unique path lifting property and $Y$ is path-connected, $\wt{f}$ is unique.
\end{proof}
We now verify that these are the \textit{only} candidates for $\lpcz$-covering maps.
\begin{lemma}\label{topologyofgencovlemma}
Suppose $p:(\hX,\hx)\to (X,x_0)$ is a based $\lpcz$-covering and $p_{\ast}(\pi_1(\hX,\hx))=H$. Then there is a homeomorphism $\wt{p}:(\hX,\hx)\to (\tXh,\txh)$ such that $p_H\circ \wt{p}=p$.
\end{lemma}
\begin{proof}
Since $p_{\ast}(\pi_1(\hX,\hx))=H\subseteq (p_H)_{\ast}(\pi_1(\tXh,\txh))$ (recall Corollary \ref{hinclusion}), there is by Lemma \ref{liftsexist} a continuous standard lift $\wt{p}:(\hX,\hx)\to (\tXh,\txh)$ such that $p_H\circ \wt{p}=p$. Suppose $[\alpha]_H\in \tXh$ and $\wt{\alpha}:(\ui,0)\to (\hX,\hx)$ is the unique lift of $\alpha$ such that $p\circ \wt{\alpha}_{\hx}=\alpha$. By definition of the standard lift, $\wt{p}(\wt{\alpha}_{\ox}(1))=[\alpha]_H$. Thus $\wt{p}$ is surjective.

Suppose $\tau_1,\tau_2:(\ui,0)\to (\hX,\hx)$ are paths such that $\wt{p}(\tau_1(1))=[p\circ \tau_1]_H=[p\circ \tau_2]_H=\wt{p}(\tau_2(1))$. Then $$p_{\ast}([\tau_1\cdot\tau_{2}^{-}])=[(p\circ \tau_1)\cdot (p\circ \tau_2)^{-}]\in H=p_{\ast}(\pi_1(\hX,\hx)).$$ By unique path lifting of $p$, it follows that $\tau_1(1)=\tau_2(1)$ and thus $\wt{p}$ is injective.

It suffices to check that $\wt{p}$ is an open map. To obtain a contradiction, suppose $V$ is an open neighborhood in $\hX$ such that $\wt{p}(V)$ is not open in $\tXh$. Pick a path $\beta:(\ui,0)\to (\hX,\hx)$ such that $\beta(1)\in V$ and no neighborhood of $\wt{p}(\beta(1))=[p\circ \beta]_H$ is contained in $\wt{p}(V)$. Let $\{U_j|j\in J\}$ be a neighborhood base at $x=p(\beta(1))$, which is directed by inclusion ($j\leq k$ if and only if $U_k\subseteq U_j$). Since $\{B([p\circ \beta]_H,U_j)|j\in J\}$ is a neighborhood base at $[p\circ \beta]_H$, there is a path $\epsilon_j:(\ui,0)\to (U_j,x)$ such that $[(p\circ \beta)\cdot \epsilon_j]_H\in B([p\circ \beta]_H,U_j)\backslash \wt{p}(V)$ for each $j\in J$. Let $x_j=\epsilon_j(1)\in U_j$.

The paths $\epsilon_j$ combine to give a map $\epsilon:(H(J),a_0)\to (X,x)$ on the arc-hedgehog. Since $H(J)$ is path-connected, locally path-connected and simply-connected, there is a unique lift $\widehat{\epsilon}:(H(J),a_0)\to (\hX,\beta(1))$ (such that $p\circ \widehat{\epsilon}=\epsilon$). Since $V$ is an open neighborhood of $\beta(1)$, there is a $k\in J$ such that $\widehat{\epsilon}$ maps the $k$-th arc of $H(J)$ in to $V$, i.e. $\widehat{\epsilon}(\ui_k)\subset V$. If $\widehat{\epsilon}_k$ is the path which is the restriction of $\widehat{\epsilon}$ to $\ui_k$, then $\widehat{\epsilon}_k(1)\in V$ and $p\circ \widehat{\epsilon}_k=\epsilon_k$. But this implies \[[(p\circ \beta)\cdot\epsilon_k]_H=[(p\circ \beta)\cdot (p\circ\widehat{\epsilon}_{k})]_H=[p\circ (\beta\cdot\widehat{\epsilon}_{k})]_H=\wt{p}(\beta\cdot\widehat{\epsilon}_{k}(1))\in \wt{p}(V)\] which is a contradiction.
\end{proof}
Combining Lemmas \ref{genunivcov} and \ref{topologyofgencovlemma}, we obtain the following theorem which, for the locally path-connected category, answers the Structure Question posed in the introduction: the topology of any $\lpcz$-covering space is equivalent to the whisker topology.
\begin{theorem}\label{equivalence1}
For any subgroup $H\subseteq \pionex$, the following are equivalent:
\begin{enumerate}
\item $H$ is a $\lpcz$-covering subgroup of $\pionex$,
\item $p_H:\tXh\to X$ is a $\lpcz$-covering,
\item $p_H:\tXh\to X$ has the unique path lifting property,
\item $(p_H)_{\ast}\left(\pi_1(\tXh,\txh)\right)=H$.
\end{enumerate}
\end{theorem}
It is not true that every $\Delta$-covering map is a $\lpcz$-covering map. For instance, if $X$ is locally path-connected but not $\Delta$-generated, then the identity map $X\to X$ is a $\lpcz$-covering map but not a $\Delta$-covering. We confirm that every $\Delta$-covering map is also a $\lpcz$-covering map when $X$ is first countable.
\begin{lemma}\label{equivalence3}
Suppose $X$ is first countable and $p:\oX\to X$ is a map. Then $p$ is a $\Delta$-covering map if and only if $p$ is a $\lpcz$-covering map.
\end{lemma}
\begin{proof}
First, suppose $p:\oX\to X$ is a $\lpcz$-covering map. By Theorem \ref{equivalence1}, $\oX\cong\tXh$ where $H=p_{\ast}(\pi_1(\oX,\ox))$. But if $X$ is first countable, then so is $\tXh$ (by the definition of the whisker topology). Since $\tXh$ is also locally path-connected, it is $\Delta$-generated by Lemma \ref{deltagenpropslemma}. Since every $\Delta$-generated space is locally path-connected, condition (2) of Definition \ref{deltacovdef} is straightforward to verify. Therefore, $p$ is a $\Delta$-covering map.

For the converse, suppose $p:\oX\to X$ is a $\Delta$-covering such that $p_{\ast}(\pi_1(\oX,\ox))=H$. Since $\oX$ is locally path-connected, $p:\oX\to X$ is continuous, and $p_{\ast}(\pi_1(\oX,\ox))=H$, there is by Lemma \ref{liftsexist}, a map $\wt{p}:\oX\to\tXh$ such that $p_H\circ\wt{p} =p$. By definition, if $\gamma:\ui\to\oX$ is a path from $\ox$ to $z$ and $\alpha=p\circ\gamma$, then $\wt{p}(\gamma(t))=\wt{\alpha}_{\mathscr{S}}(t)=[\alpha_t]_H$ where $\alpha_t(s)=\alpha(st)$ is the path given by restricting $\alpha$ to $[0,t]$. We show in the next paragraph that $\wt{p}$ is a homeomorphism. Once this is done, it is clear that $p_H$ has the unique path lifting property and is therefore a $\lpcz$-covering (by Theorem \ref{equivalence1}) which is equivalent to $p$.

Our proof that $\wt{p}$ is a homeomorphism is quite similar to the proof of Lemma \ref{topologyofgencovlemma}. Verifying bijectivity is identical so we leave it to the reader. We check that $\wt{p}$ is open: To obtain a contradiction, suppose $V$ is open in $\oX$ and $\wt{p}(V)$ is not open in $\tXh$. Pick a path $\beta:(\ui,0)\to (\oX,\ox)$ such that $\beta(1)\in V$ and no neighborhood of $\wt{p}(\beta(1))=[p\circ \beta]_H$ is contained in $\wt{p}(V)$. Let $\{U_n|n\geq 1\}$ be a countable neighborhood base at $x=p(\beta(1))$. Since $\{B([p\circ \beta]_H,U_n)|n\geq 1\}$ is a neighborhood base at $[p\circ \beta]_H$ and $\wt{p}(V)$ is not open, there is a path $\epsilon_n:(\ui,0)\to (U_n,x)$ such that $[(p\circ \beta)\cdot \epsilon_n]_H\in B([p\circ \beta]_H,U_n)\backslash \wt{p}(V)$ for each $n\geq 1$. Let $x_n=\epsilon_n(1)\in U_n$.

The paths $\epsilon_n$ combine to give a map $\epsilon:(H(\omega),a_0)\to (X,x)$ on the $\omega$-arc-hedgehog. Since $H(\omega)$ is $\Delta$-generated and simply-connected, there is a unique lift $\overline{\epsilon}:(H(\omega),a_0)\to (\oX,\beta(1))$ (such that $p\circ \overline{\epsilon}=\epsilon$). Since $V$ is an open neighborhood of $\beta(1)$, there is a $N$ such that $\overline{\epsilon}$ maps the $n$-th arc of $H(\omega)$ in to $V$ for all $n\geq N$. If $\overline{\epsilon}_n$ is the path which is the restriction of $\overline{\epsilon}$ to $\ui_n$, then $\overline{\epsilon}_n(1)\in V$ and $p\circ \overline{\epsilon}_n=\epsilon_n$. But this implies \[[(p\circ \beta)\cdot\epsilon_n]_H=[(p\circ \beta)\cdot (p\circ\overline{\epsilon}_{n})]_H=[p\circ (\beta\cdot\overline{\epsilon}_{n})]_H=\wt{p}(\beta\cdot\overline{\epsilon}_{n}(1))\in \wt{p}(V)\] for some $n$, which is a contradiction.
\end{proof}
Combining the previous lemma with Theorem \ref{equivalence1}, we obtain the following comparison for first countable spaces. This comparison states implies that for a first countable space $X$, the subgroup structure of $\pionex$ retained by all disk-coverings, $\Delta$-coverings, and $\lpcz$-coverings is exactly the same.
\begin{theorem}\label{equivalence2}
If $X$ is first countable, then $\mathbf{Cov}_{\lpcz}(X)=\delcovx$. Moreover, the following diagram of functors commutes up to natural isomorphism.
\[\xymatrix{
\dcovx \ar[dr]_-{\mu} \ar[r]^-{\Delta} & \delcovx \ar@{^{(}->}[d]^-{\mu}  \ar@{=}[r] & \mathbf{Cov}_{\lpcz}(X) \ar@{^{(}->}[dl]^-{\mu}\\
& G\mathbf{Set}
}\]
\end{theorem}
The study of wild fundamental groups is typically focused on the fundamental groups of metric spaces. The previous theorem implies that any covering-theoretic attempt to characterize the subgroup structure of the fundamental group of a metric space is retained or improved upon by the $\lpcz$-coverings introduced in \cite{FZ07}. In particular, there can be no mysterious disk-covering subgroup $H\subseteq\pionex$ for which the classical endpoint projection $p_H:\tXh\to X$ fails to have unique path lifting. The author does not know of an example of a non-first countable, $\Delta$-generated space $X$ and subgroup $H\subseteq \pionex$ such that $H$ is a $\Delta$-covering subgroup but not a $\lpcz$-covering subgroup.
\begin{example}\label{failureexample}
Let $\bbh=\bigcup_{n\geq 1}\{(x,y)\in\bbr^2|(x-1/n)^2+y^2=1/n^2\}$ be the standard Hawaiian earring space and $\ell_n:S^1\to \bbh$ be a simple closed curve traversing the n-th circle. We consider the subgroup $F_{\infty}\subseteq \pi_1(\bbh,(0,0))$ freely generated by the homotopy classes of the loops $\ell_n$. There does not exist any subcategory $\mcc\subseteq \topz$ such that $F_{\infty}$ is a $\mcc$-covering subgroup. If such a subcategory were to exist, then $F_{\infty}$ would be a $\Delta$-covering subgroup and therefore a $\lpcz$-subgroup since $\bbh$ is first countable. But by Theorem \ref{equivalence1}, this would imply that $p_{F_{\infty}}:\tX_{F_{\infty}}\to X$ has unique path-lifting, which is false: see Example 6.2 of \cite{FZ07}
\end{example}
\begin{example}\label{failureexample2}
For an even more extreme example consider any non-simply connected, first countable space $X$ such every non-trivial homotopy class has a representative in an arbitrary neighborhood of the basepoint $x_0$. Such spaces are considered in detail in \cite{Virk1,Virk2} and include the famous harmonic archipelago and Griffiths twin cone spaces. It is well-known that the only subgroup $H\subseteq \pionex$ for which $p_H:\tXh\to X$ has the unique path lifting property is when $H=\pionex$. Using the same reasoning as in the previous example, we can conclude that given any $\mcc$, the only $\mcc$-covering subgroup of $\pionex$ is $\pionex$ itself. Thus one cannot possibly distinguish any elements of $\pionex$ using any theory of generalized coverings.
\end{example}
\section{Continuous lifting: $\mathbf{Fan}$-coverings}
\begin{definition}
Let $(J,\leq)$ be a directed set and $K=J\cup\{\infty\}$ be obtained by adding one maximal point. Give $K$ the topology generated by basis sets $\{k\}$ and $V_k=\{\infty\}\cup\{j\in J|j>k\}$ for $k<\infty$. The \textit{directed arc-fan over} $J$ is the cone over $K$, i.e. the quotient space $F(J)=K\times \ui/K\times\{0\}$ with basepoint $b_0$, the image of $K\times \{0\}$ in the quotient.
\end{definition}
We denote the $k$-th arc of $F(J)$ (i.e. the image of $\{k\}\times\ui$) by $\ui_k$ and take the point $t_k$ to be the image of $\{(k,t)\}$ in $\ui_k$. The restriction of a map $f:F(J)\to X$ to $\ui_k$ is a path $f_k:\ui_{k}\to X$. Observe that maps $f:(F(J),b_0)\to (X,x)$ are in bijective correspondence with convergent nets $f_j\to f_{\infty}$ in $P(X,x)$ indexed by $J$.

Let $\mcf$ be the category whose objects are $D^2$ and the directed fans $F(J)$ for all directed sets $J$. Let $\fan=\scrh(\mcf)$ be the coreflective hull of $\mcf$ in $\topz$. Note that the arc-hedgehog $H(J)=F(J)/(\{\infty\}\times\ui)$ is a quotient of $F(J)$ and thus $\lpcz\subset \fan$. The directed fan $F(\omega)$ is a non-locally path-connected continuum. Thus $\lpcz$ is a proper subcategory of $\fan$.
\begin{lemma}\label{quotienteval}
A path-connected space $X$ is an object of $\fan$ if and only if for some (equivalently any) $x\in X$, the map $ev:P(X,x)\to X$, $ev(\alpha)=\alpha(1)$ is quotient.
\end{lemma}
\begin{proof}
First note that if $y\in X$ and $\beta:\ui\to X$ is any path from $y$ to $x$, then $ev:P(X,x)\to X$ factors as the composition of the left concatenation map $\lambda_{\beta}:P(X,x)\to P(X,y)$, $\lambda_{\beta}(\alpha)=\beta\cdot\alpha$ and evaluation $ev:P(X,y)\to X$. Thus if $ev:P(X,x)\to X$ is quotient, then $ev:P(X,y)\to X$ is quotient for all $y\in X$.

Suppose $ev:P(X,x)\to X$ is quotient for all $x\in X$. We show that $X\in\fan$. Suppose $U\subseteq X$ is a set such that for every directed set $J$ and map $f:F(J)\to X$, $f^{-1}(U)$ is open in $F(J)$. Fix $x\in X$. We check that $ev^{-1}(U)$ is open in $P(X,x)$. If $ev^{-1}(U)$ is not open, then there is a convergent net $f_{j}\to f_{\infty}$, $j\in J$ in $P(X,x)$ such that $f_{\infty}\in ev^{-1}(U)$ and $f_j\notin ev^{-1}(U)$ for every $j\in J$. This net uniquely induces a map $f:(F(J),b_0)\to(X,x)$ for which $f^{-1}(U)$ is open by assumption. But $1_{\infty}\in [0,1]_{\infty}\cap f^{-1}(U)$ which implies the existence of $j_0$ such that $1_{j}\in  f^{-1}(U)$ for all $j>j_0$. This contradicts the assumption that $f_j(1)=f(1_j)\notin U$ for all $j\in J$. Therefore, $ev^{-1}(U)$ must be open in $P(X,x)$. Since $ev$ is quotient, $U$ is open in $X$. This proves $X\in\fan$.

For the converse, suppose $X\in\fan$, $x\in X$, and $U\subseteq X$ is such that $ev^{-1}(U)$ is open in $P(X,x)$. To see that $ev$ is quotient, we check that $U$ is open. If $U$ is not open, then there is a map $f:(F(J),b_0)\to (X,y)$ on a directed arc-fan such that $f^{-1}(U)$ is not open in $F(J)$. Pick a point $z\in f^{-1}(U)$ such that no neighborhood of $z$ is contained in $f^{-1}(U)$.

Case I: Suppose $z=b_0$. Since $F(J)$ is locally path-connected at $b_0$, there is a directed set $K$ and a map $h:(H(K),a_0)\to (F(J),b_0)$ such that $h(1_k)\notin f^{-1}(U)$ for every $k\in K$. But since $H(K)\in\fan$ and $f\circ h:(H(K),a_0)\to (X,y)$ is continuous, $h^{-1}(f^{-1}(U))$ is an open neighborhood of $a_0$ in $H(K)$. Thus there is a $k_0\in K$ such that $1_k\in  h^{-1}(f^{-1}(U))$ for all $k>k_0$. This contradicts the fact that $h(1_k)\notin f^{-1}(U)$ for every $k\in K$. Thus $U$ must be open.

Case II: Suppose $z\neq b_0$. There is a directed set $L$ and a net $z_{\ell}\to z_{\infty}=z$, ${\ell}\in L$ in $F(J)$ such that $z_{\ell}\notin f^{-1}(U)$ for every ${\ell}\in L$. Since $z\neq b_0$, we may assume $z_{\ell}$ has image in the open set $F(J)\backslash \{b_0\}\cong (J\cup\{\infty\})\times (0,1]$. Thus for each $\ell\in L\cup\{\infty\}$, we have $z_{\ell}=t_{j(\ell)}\in (0,1]_{j(\ell)}$ for some $j(\ell)\in J$. Define a function $h:(F(L),b_0)\to (F(J),b_0)$ so that for all $\ell\in L\cup \{\infty\}$, the restriction of $h$ to $[0,1]_{\ell}$ is the linear path $h(s_{\ell})=(st)_{j(\ell)}$. Since $[0,1]$ is locally compact and $L\cup \{\infty\}$ is Hausdorff, the continuity of $h$ follows directly from exponential laws for maps on pairs of spaces. Note that $h(1_{\ell})=z_{\ell}$ for all $\ell\in L\cup\{\infty\}$. Since $f$ and $h$ are continuous, $(f\circ h)_{\ell}\to (f\circ h)_{\infty}$ is a convergent net of paths in $P(X,y)$. Now if $\gamma$ is any path in $X$ from $x$ to $y$, then $\gamma\cdot (f\circ h)_{\ell}\to \gamma\cdot (f\circ h)_{\infty}$ in $P(X,x)$. Since $\gamma\cdot (f\circ h)_{\infty}(1)=f(z)\in U$, we have $\gamma\cdot (f\circ h)_{\infty} \in ev^{-1}(U)$. By assumption, $ev^{-1}(U)$ is open so there is an $\ell_0$ such that $\gamma\cdot (f\circ h)_{\ell}\in ev^{-1}(U)$ for all $\ell>\ell_0$. However, this contradicts the fact that $\gamma\cdot (f\circ h)_{\ell}(1)=f(z_\ell)\notin U$ for every $\ell\in L$. Thus $U$ must be open.
\end{proof}
\begin{corollary}
Suppose $fan:\topz\to\fan$ is the coreflection functor. The topology of $fan(X)$ is the quotient topology on $X$ with respect to evaluation $ev:P(X,x)\to X$ for any $x\in X$.
\end{corollary}
\begin{proof}
Pick any $x\in X$. The continuous identity $fan(X)\to X$ is a disk-covering and therefore induces a continuous bijection $b:P(fan(X),x)\to P(X,x)$ on path spaces. A convergent net $\alpha_{j}\to \alpha$, $j\in J$ in $P(X,x)$ corresponds to a continuous map $f:F(J)\to X$ on a directed arc-fan for which $f:F(J)\to fan(X)$ is also continuous since $F(J)\in\fan$. Thus $\alpha_{j}$ also converges to $\alpha$ in $P(fan(X),x)$. Consequently, $b$ is a homeomorphism. By Lemma \ref{quotienteval}, $ev:P(fan(X),x)\to fan(X)$ is quotient. Composing $b^{-1}$ with this quotient map completes the proof.
\end{proof}
\begin{theorem}\label{contractible}
Every contractible space is an object of $\fan$. Consequently, if $\mathbf{Cntr}$ is the category of contractible spaces, then $\scrh(\mathbf{Cntr})=\fan$.
\end{theorem}
\begin{proof}
Since every directed arc-fan is the cone over a convergent net, we have $\mcf\subset \mathbf{Cntr}$ and thus $\fan\subset\scrh(\mathbf{Cntr})$. Recall that every contractible space $Z$ is the retract (and therefore quotient) of the cone $CZ=Z\times \ui/Z\times \{0\}$ over $Z$. Since coreflective hulls are closed under quotients, it suffices to show $CZ\in \fan$ for any (possibly non-path-connected) space $Z$. Let $t_z$ denote the image of $(z,t)$ in $CZ$ and $v=0_z$ be the vertex. Using the exponential law for based path spaces it is straightforward to verify that the map $f:CZ\to P(CZ,v)$, $f(t_z)(s)=(st)_{z}$ is continuous. If $ev:P(CZ,v)\to CZ$ is evaluation, then $ev\circ f=id$, showing that $ev$ is a retraction and therefore a quotient map. It follows from Lemma \ref{quotienteval} that $CZ\in \fan$.
\end{proof}
\begin{definition}\cite{Brazsemi}
A map $f:X\to Y$ has \textit{continuous lifting of paths} if for every $x\in X$, the induced map $f_{\#}:P(X,x)\to P(Y,f(x))$ is a homeomorphism.
\end{definition}
\begin{proposition}
The following are equivalent for a map $p:\widehat{X}\to X$:
\begin{enumerate}
\item $p$ has continuous lifting of paths,
\item $p$ is a weak $\mcf$-covering map,
\item $p$ is a weak $\mathbf{Cntr}$-covering map.
\end{enumerate}
\end{proposition}
\begin{proof}
(1) $\Leftrightarrow$ (2) Suppose $\hat{x}\in\widehat{X}$, and $p(\hat{x})=x$. First observe that $p_{\#}:P(\widehat{X},\hat{x})\to P(X,x)$ is a homeomorphism if and only if every map $f:(F(J),b_0)\to (X,x)$ on a directed arc-fan lifts uniquely to a map $\hat{f}:(F(J),b_0)\to (\widehat{X},\hat{x})$ such that $p\circ\hat{f}=f$. Since directed arc-fans are simply connected, it is immediate that every $\mcf$-covering has continuous lifting of paths. It follows directly from exponential laws of mapping spaces that every map with continuous lifting of paths also is a disk-covering \cite[Remark 2.5]{Brazoverlay}.

(2) $\Leftrightarrow$ (3) We prove the more general fact that if $\hullc=\hulld$, then $p:\hX\to X$ is a weak $\mcc$-covering if and only if $p$ is a weak $\mcd$-covering. By symmetry, we only need to prove one direction. Suppose $p$ is a weak $\mcc$-covering and $f:(Y,y)\to (X,x)$ is a map with $Y\in \mcd$ and $f_{\ast}(\pi_1(Y,y))\subseteq p_{\ast}(\pi_1(\hX,\hx))$. Since $p$ is a disk-covering it suffices to show the unique function $\widehat{f}:(Y,y)\to (\hX,\hx)$ such that $p\circ \widehat{f}=f$ is continuous. If $c:\topz\to\hullc$ is the coreflection, then by Lemma \ref{ccovtohullccovlemma} $\psi(p):c(\hX)\to X$ is a $\hullc$-covering. Since $Y\in \hulld=\hullc$, the unique lift $\widehat{f}:(Y,y)\to (c(\hX),\hx)$ is continuous. Composing this map with the continuous identity $id:c(\hX)\to\hX$ shows that $\widehat{f}:Y\to \hX$ is continuous.
\[\xymatrix{
c(\widehat{X}) \ar[drr]^-{\psi(p)} \ar[rr]^-{id} && \widehat{X} \ar[d]^-{p}\\ Y \ar[rr]_-{f} \ar@{-->}[u]^-{\widehat{f}} && X
}\]
\end{proof}
\begin{corollary}\label{contlifting2}
A map $p:\widehat{X}\to X$ is a $\fan$-covering map if and only if $\widehat{X}\in\fan$ and $p$ has continuous lifting of paths.
\end{corollary}
%
%
In the previous section, we identified the topology of $\lpcz$-coverings as the whisker topology. In the same fashion, we identify the topology of $\fan$-coverings as the natural quotient topology. Let $\tXh^{qtop}$ denote the space with the same underlying set of $\tXh$ (used in the previous section) with the quotient topology with respect to the map $q:P(X,x_0)\to \tXh$, $q(\alpha)=[\alpha]_H$.

Using exponential properties of mapping spaces, it is shown in \cite{Brazsemi} that the endpoint projection $p_H:\tXh^{qtop}\to X$, $p_H([\alpha]_H)=\alpha(1)$ induces a retraction $(p_H)_{\#}:P(\tXh^{qtop},\txh)\to P(X,x_0)$ on path spaces. The canonical section $s:P(X,x_0)\to P(\tXh^{qtop},\txh)$ is given by $s(\alpha)=\wt{\alpha}_{\mathscr{S}}$, the standard lifts of paths starting at $\txh$. Thus the inclusion $H\subseteq(p_H)_{\ast}\left(\pi_1\left(\tXh^{qtop},\txh\right)\right)$ holds in general. Since a retraction is a homeomorphism if and only if it is injective, the next lemma follows immediately.
\begin{lemma}
\cite[Theorem 7.7]{Brazsemi} The endpoint projection $p_H:\tXh^{qtop}\to X$, $p_H([\alpha]_H)=\alpha(1)$ has continuous lifting of paths if and only if it has the unique path lifting property.
\end{lemma}
\begin{lemma}\label{liftsexist2} Suppose $Y\in\fan$, $y_0\in Y$, $\alpha\in P(X,x_0)$, and $f:(Y,y_0)\to (X,\alpha(1))$ is a map such that $f_{\ast}(\pi_1(Y,y_0))\subseteq [\alpha^{-}]H[\alpha]$. Then there is a continuous map $\wt{f}:(Y,y_0)\to (\tXh^{qtop},[\alpha]_H)$ such that $p_H\circ \wt{f}=f$ defined by $\wt{f}(y)=[\alpha\cdot (f\circ \tau)]_H$ where $\tau:\ui\to Y$ is any path from $y_0$ to $y$.
\end{lemma}
\begin{proof}
As in Lemma \ref{liftsexist}, $\wt{f}$ is well-defined; we are left to verify continuity. Consider the following commuting diagram.
\[\xymatrix{
 P(Y,y_0) \ar[d]_-{ev}  \ar[r]^-{f_{\#}} & P(X,f(y_0))  \ar[r]^-{\lambda_{\alpha}} & P(X,x_0)  \ar[d]^-{q}\\
Y \ar[rr]_-{\wt{f}} && \tXh^{qtop}
}\]The map $\lambda_{\alpha}(\beta)=\alpha\cdot\beta$ is left concatenation by $\alpha$. Each function in the top composition is continuous and the left vertical map is a quotient map by Lemma \ref{quotienteval}. Thus $\wt{f}$ is continuous by the universal property of quotient spaces.
\end{proof}
\begin{lemma}\label{fanhomeomorphismlemma}
If $p:\widehat{X}\to X$ is a $\fan$-covering and $H=p_{\ast}(\pi_1(\hX,\hx))\subseteq \pionex$, then there is a homeomorphism $\hat{p}:\hX\to\tXh^{qtop}$ such that $p_H\circ\hat{p}=p$.
\end{lemma}
\begin{proof}
As in Lemma \ref{topologyofgencovlemma}, there is a canonical bijection $\hat{p}:\hX\to\tXh^{qtop}$ defined as follows. If $\gamma$ is any path in $\hX$ starting at $x_0$, then $\hat{p}(\gamma(1))=[p \circ \gamma]_H$. Consider the following diagram which commutes since if $\gamma\in P(\hX,\hx)$, then $\hat{p}(ev(\gamma))=\hat{p}(\gamma(1))=[p\circ\gamma]_H=q(p_{\#}(\gamma))$.
\[\xymatrix{
P(\hX,\hx) \ar[r]_-{\cong}^-{p_{\#}} \ar[d]_-{ev} & P(X,x_0)  \ar[d]^{q}\\
\hX \ar[r]_-{\hat{p}} & \tXh^{qtop}
}\]Since $p$ has continuous lifting of paths (Corollary \ref{contlifting2}), the top map is a homeomorphism. Since $\hX\in\fan$, the left vertical map is quotient by Lemma \ref{quotienteval}. The right vertical map is quotient by construction. Therefore the bottom horizontal map must be a continuous quotient map. Since $\hat{p}$ is bijective, it is a homeomorphism.
\end{proof}
As we did for the locally path-connected case, we can now answer the Structure Question for the category $\fan$: the topology of any $\fan$-covering space is equivalent to the natural quotient topology.
\begin{theorem}\label{equivalencequotienttop}
For any subgroup $H\subseteq \pionex$, the following are equivalent:
\begin{enumerate}
\item $H$ is a $\fan$-covering subgroup of $\pionex$,
\item $p_H:\tXh^{qtop}\to X$ is a $\fan$-covering,
\item $p_H:\tXh^{qtop}\to X$ has the unique path lifting property,
\item $(p_H)_{\ast}\left(\pi_1(\tXh^{qtop},\txh)\right)=H$.
\end{enumerate}
\end{theorem}
\begin{proof}
Using Lemma \ref{liftsexist2}, the proof of (2) $\Leftrightarrow$ (3) $\Leftrightarrow$ (4) is identical to the proof of Lemma \ref{genunivcov}. (1) $\Rightarrow$ (2) follows directly from Lemma \ref{fanhomeomorphismlemma}. (2) \& (4) $\Rightarrow$ (1) is by definition.
\end{proof}
\begin{remark}
In general, it is known that the whisker topology on $\tXh$ is finer than the quotient topology of $\tXh^{qtop}$ \cite{VZ14compare}. Our characterization of the topology of $\lpcz$- and $\mathbf{Fan}$-coverings indicates that if $H$ is a $\fan$-covering subgroup of $\pionex$, then $lpc\left(\tXh^{qtop}\right)\cong \tXh$.
\end{remark}
\end{document}